\documentclass[12pt]{amsart} 

\usepackage{graphicx,amssymb,colordvi,textcomp,latexsym} 

\usepackage{tikz}
\usepackage{ulem}

\usepackage[title]{appendix}
\usepackage[showframe=false]{geometry}
\usepackage{changepage}

\usepackage{tikz}
\usetikzlibrary{arrows,shapes, positioning, matrix, patterns, decorations.pathmorphing, arrows.meta,automata}

\usepackage{color}
\usepackage{colortbl}
\usepackage{soul}

\usepackage{array}
\usepackage[hidelinks]{hyperref}
\usepackage{amsfonts}
\usepackage{amsmath}
\usepackage{amsthm}

\usepackage{adjustbox}
\usetikzlibrary{arrows,shapes, positioning, matrix, patterns, decorations.pathmorphing}

\newcommand{\real}{{\mathbb{R}}}
\newcommand{\R}{\mbox{$\mathbb{R}$}}

\newcommand{\ZZ}{\mathbf{Z}}

\newcommand{\NN}{\mathbf{N}}

\newtheorem{lemma}{Lemma}[section]

\newtheorem{prop}[lemma]{Proposition}
\newtheorem{thm}[lemma]{Theorem}
\newtheorem{cor}[lemma]{Corollary}
\newtheorem{alg}[lemma]{Algorithm}
\theoremstyle{definition}
\newtheorem{Def}[lemma]{Definition}
\newtheorem{exam}[lemma]{Example}

\theoremstyle{remark}
\newtheorem{rem}[lemma]{Remark}

\title[]{Classification of networks with asymmetric inputs}
\author{Manuela Aguiar}
\address{Manuela Aguiar, Faculdade de Economia, Centro de Matem\'atica, Universidade do Porto,
Rua Dr Roberto Frias, 4200-464 Porto, Portugal.}
\email{maguiar@fep.up.pt}
\author{Ana Dias}
\address{Ana Dias, Departamento de Matem\'atica, Centro de Matem\'atica, Universidade do Porto,
Rua do Campo Alegre, 687, 4169-007 Porto, Portugal.}
\email{apdias@fc.up.pt}
\author{Pedro Soares}
\address{Pedro Soares, Faculty of Applied Physics and Mathematics, Gda\'nsk University of Technology, Narutowicza 11/12, 80-233
Gdańsk, Poland; Centro de Matem\'atica, Universidade do Porto, Rua do Campo Alegre, 687, 4169-007 Porto, Portugal.}
\email{pedro.soares@pg.edu.pl}

\thanks{Corresponding author: pedro.soares@pg.edu.pl}

\begin{document}

\begin{abstract} Coupled cell systems associated with a coupled cell network are determined by (smooth) vector fields that are consistent with the network structure. Here, we follow the formalisms of  Stewart, Golubitsky and Pivato (Symmetry groupoids and patterns of synchrony in coupled cell networks, {\it SIAM J. Appl. Dyn. Syst.} {\bf 2} (4) (2003) 609--646), Golubistky, Stewart and T\"{o}r\"{o}k (Patterns of synchrony in coupled cell networks with multiple arrows,  {\it SIAM J. Appl. Dynam. Sys.} {\bf 4} (1) (2005) 78--100) and Field (Combinatorial dynamics, {\it Dynamical Systems} {\bf 19} (2004) (3) 217--243). It is known that two non-isomorphic $n$-cell coupled networks can determine the same sets of vector fields --  these networks are said to be ODE-equivalent. The set of all $n$-cell coupled networks is so partitioned into classes of ODE-equivalent networks. With no further restrictions, the number of ODE-classes is not finite and each class has an infinite number of networks. Inside each ODE-class we can find a finite subclass of networks that minimize the number of edges in the class, called minimal networks.

In this paper, we consider coupled cell networks with asymmetric inputs. That is, if $k$ is the number of distinct edges types, these networks have the property that every cell receives $k$  inputs, one of each type. Fixing the number $n$ of cells, we prove that: the number of ODE-classes is finite; restricting to a maximum of $n(n-1)$ inputs, we can cover all the ODE-classes; all minimal $n$-cell networks with $n(n-1)$ asymmetric inputs are ODE-equivalent. We also give a simple criterion to test if a network is minimal and we conjecture lower estimates for the number of distinct ODE-classes of $n$-cell networks with any number $k$ of asymmetric inputs. Moreover, we present a full list of representatives of the ODE-classes of networks with three cells and two asymmetric inputs.
\end{abstract}

\date{\today}


%34C20 Ordinary differential equations | Qualitative theory | Transformation and reduction of equations and systems, normal forms 
%05C90 Combinatorics |  Graph theory | Applications 
%05C30 Combinatorics  | Graph theory | Enumeration in graph theory 
%15A36 Linear and multilinear algebra; matrix theory  | Basic linear algebra | Matrices of integers 

\ 

\vspace{20mm}

\maketitle

\tableofcontents

\section{Introduction}

In this paper, we follow the formalisms of Stewart, Golubitsky and Pivato~\cite{SGP03}, Golubitsky, Stewart and T\"{o}r\"{o}k~\cite{GST05} and Field~\cite{F04} where a {\it (coupled cell) network} is represented by a directed graph,  representing schematically a set of dynamical systems (the cells) and their dependencies (the couplings). The nodes of the graph represent the cells and the edges represent the couplings. Any such directed graph on $n$ nodes can be represented by a set of adjacency matrices, one for each type of input. 
Here we consider that each cell represents a system of ordinary differential equations (ODEs) where auto-couplings (self-loops) are allowed. Note that the number of networks grows exponentially with the number of cells and the number of edges.  See, for example, Aldosray and Stewart~\cite{AS05} for the enumeration of networks with a single type of cell and a single type of input such that all the cells receive the same number of inputs.

A {\it coupled cell system} associated with a network has to respect its topology. 
That is, choosing cell phase spaces, the network admissible vector fields are the smooth mappings on the total phase space that reflect whose cells are coupled to whom, and whose cells and couplings are identical or not. Golubitsky, Stewart and T\"{o}r\"{o}k~\cite{GST05} have remarked  that non-isomorphic networks can correspond to the same space of admissible vector fields, that is, be {\it ODE-equivalent}.
In particular, it follows that the sets of dynamics that can be realized by ODE-equivalent networks are the same, which from the modelling point of view is important in its own.  
Dias and Stewart~\cite{DS05} showed that two $n$-cell networks are ODE-equivalent if and only if are {\it linearly equivalent}, choosing cell phase spaces to be $\R$. 
That is,  considering the real linear space $M$ of the $n \times n$ matrices with real entries, two $n$-cell networks are ODE-equivalent if and only if, renumbering the cells  of one of the networks if necessary, the corresponding real linear subspaces of $M$ generated by the associated  adjacency matrices coincide.
Using this result, Aguiar and Dias~\cite{AD07} characterize  the subclass of any ODE-class of networks with minimal number of edges, including an algorithm for obtaining those {\it minimal} subclasses. The number of networks at each ODE-class is not finite, however the subclass of minimal networks is finite.

%Networks with asymmetric inputs(Section 1)
We restrict our attention to {\it networks with $k$ asymmetric inputs}, where $k$ is any positive integer number. 
That is, networks with $k$ input types and where each node receives exactly one input of each type. We remark that these networks are {\it homogeneous}, that is, every cell receives exactly the same number of inputs. It is proved in Aguiar, Ashwin, Dias and Field~\cite{AADF11} that these networks can support robust heteroclinic cycles, even in low dimension. See also Aguiar~\cite{A18} for the synchrony lattice of networks with asymmetric inputs. 
In Aguiar, Dias and Soares~\cite{ADS19}, it is studied the steady-state lifting bifurcation problem for those networks. All the theory developed on normal form and bifurcation theory by Rink and Sanders~\cite{RS14,RS15} and Nijholt, Rink and Sanders~\cite{NRS16}-\cite{NRS17b} concerns networks with asymmetric inputs. 

%Our methodology(section 4)
We present a methodology for classifying  networks of $n$-cells with $k$ asymmetric inputs. Fixing the number $n$ of cells, we ask if  the number of distinct ODE-classes of networks  is finite and, in that case, if there is a methodology of enumerating minimal $n$-cell  networks with $k$ asymmetric inputs, for any $k$. Following Aguiar and Dias~\cite{AD07}, we provide a simple criterion to test if a network is minimal: an $n$-cell network with $k$ asymmetric inputs is minimal if and only if the $n \times n$ identity matrix and the corresponding $k$ adjacency matrices are linearly equivalent (Proposition~\ref{prop:min}).

% Motiviation works for section 5
In view of the large number of possible networks, different authors have focus their attention to classify and study networks with a  low number of cells and inputs.
Leite and Golubistky~\cite{LG06} classify all three-cell networks with identical cells and couplings, i.e., just one cell type and one input type  with valency one or two. 
They show that, up to ODE-equivalence, there are 34 distinct connected such networks.
Aguiar, Ashwin, Dias and Field~\cite{AADF11} presented the $10$ ODE-classes of strongly connected networks with three cells, two asymmetric inputs and one or two two-dimensional synchrony subspaces.
Rink and Sanders~\cite{RS14} classified the homogenous networks with two and three cells and asymmetric inputs which have a monoid symmetry.

% Motiviation works for section 6 and 7
Small networks are also called motifs networks, \cite{MSIKCA02}.
Motifs networks can be viewed as building blocks of complex networks.
The frequency of these motifs networks in complex networks can reveal some characteristics of those complex networks.
In particular, some studies suggest that the motifs' frequency is related with the function and the context of the complex networks such as biological or social networks.
In the formalism of coupled cell systems, Golubitsky, Stewart and T\"{o}r\"{o}k~\cite{GST05} noted the existence of invariant subspaces given by the synchronization of some cells and thus called synchrony subspaces.
Those synchrony subspaces only depend on the network structure and are independent from the given coupled cell system.
Moreover, the restriction of the system to a synchrony subspace corresponds to a coupled cell system in a smaller network.
Therefore the study of small networks can help our understanding of bigger networks.

%Our results(Section 5)
Motivated by the works mentioned above, we consider to be of interest, as a start, to enumerate all three-cell networks with identical cells and one or two types of input
 where each cell receives exactly one coupling of each type.
 In a follow-up work, we study the steady-state bifurcation problems for these networks~\cite{ADS19b}. There are $650$ networks with three cells and two asymmetric inputs that we reduce to a list of $48$ minimal networks representing all the different ODE-classes (see Theorem~\ref{thm:enumeration} and Tables~\ref{tab:C3L0RI2.tex}-\ref{tab:C2L3.tex}).
Reducing the list of all  three-cell networks to minimal representative networks is a demanding computational task, because we need to compare the linear vector spaces generated by the adjacency matrices of two networks and all possible permutations of the cells of one of the networks.
Note that this list includes the $10$ classes of networks presented by Aguiar, Ashwin, Dias and Field~\cite{AADF11} and the seven $3$-cells networks considered by Rink and Sanders~\cite{RS14}.
Surprisingly, as we remark, two $3$-cell non-ODE-equivalents have the same monoid symmetry with $3$ elements.
The particular case of the three-cell  networks with two asymmetric inputs already illustrates the difficulty and the amount of work involved to classify networks, up to ODE-equivalence. 

%Section 6, maximal inputs n(n-1)

 In this work, we prove that the maximum number of asymmetric inputs for a minimal network with $n$ cells is $n(n-1)$ (Theorem~\ref{thm:dim}). 
Thus any $n$-cell  network with asymmetric inputs is ODE-equivalent to an $n$-cell  network with at most $n(n-1)$ asymmetric inputs. That is, the set of dynamics that can occur for  $n$-cell networks with at most $n(n-1)$ asymmetric inputs covers all possible types of dynamics that can occur for any  $n$-cell network with any number of asymmetric inputs.  
Moreover, we remark that all minimal  networks of $n$-cells with $n(n-1)$ asymmetric inputs are ODE-equivalent (Corollary~\ref{cor:minimal}). 
%Section 7
We present a minimal network with $n(n-1)$ asymmetric inputs which represents such unique ODE-class (Theorem~\ref{cor:rep_min}).
Surprisingly this representative is given by the union of $n(n-1)$ feed-forward networks with one input.
{\it In feed-forward networks}, cells are arranged in layers, and the information moves only in one direction, forward, from the input nodes (first layer), through the hidden nodes (middle layers), and to the output nodes (last layer). 
This class of networks have been applied in different fields and theoretical studies of these kind of networks have been addressed. 
See for example~\cite{ADF17,S18,ADF19} and references therein to specific applications. 
Feed-forward systems can  exibit dynamical features that are not common in systems without feed-forward structure.
One example is the occurrence of generic Hopf bifurcation in one-parameter families of coupled cell systems, from an equilibrium to periodic solutions and where there is growth of the amplitude of cells (as a function of the bifurcation parameter) faster than would be expected in systems that do not have the feed-forward structure~\cite{RS13}. 
See also \cite{NRS17a,S18} where a similar phenomenon is proved in the steady-state bifurcation case for feed-forward systems and for more recent work ~\cite[Chapters 8-11]{SS19} and \cite{NRSS19}.

%Consequences of Section 6
Our results imply that, fixing the number $n$ of cells, the number of distinct ODE-classes of  $n$-cell networks with any number of asymmetric inputs is finite.  
Therefore, we can obtain a finite list of minimal networks such that each ODE-class of  networks with $n$-cells and asymmetric inputs is uniquely represented.
In order to list every ODE-class we can repeat the method we use in Section~\ref{sec:casestudy} for three cell networks and enumerate all distinct ODE minimal $n$-cell  networks with $k$ asymmetric inputs, where $k$ runs from $1$ to $n(n-1)$.
Alternatively, we can start with a list of every network with $n(n-1)$ asymmetric inputs and then reduce it to a list of minimal representative networks.
%Section 8 
We present two algorithms to construct minimal representative networks with $n$ cells and one asymmetric input of distinct ODE-classes using the minimal representatives networks with less than $n$ cells. 
These algorithms allow us to prove that there are at least $n(n-1)$ ODE-classes where the minimal networks have $n$ cells and one asymmetric input (Theorem~\ref{thm:min_ncell_1input}).

The manuscript is organised as follows: Section~\ref{sec:pre} establishes some definitions that are used in the rest of the paper. In Section~\ref{sec:ODE} we recall the structure of coupled cell systems consistent with networks. The definition of ODE-equivalence of networks is given and it is stated the result that establishes that ODE-equivalent networks are the linear equivalent networks. The minimality of networks is defined in Section~\ref{sec:criterion}, together with a criterion is given for checking the minimality of  networks with asymmetric inputs.  This simple criterion is a basic tool used in Section~\ref{sec:main} to obtain our main results of the paper. 
In Section~\ref{sec:casestudy}, we illustrate the ideas described above, by presenting the classification of the three-cell  networks with two asymmetric inputs. Our general results concerning the classification of $n$-cell  networks with asymmetric inputs appear in Section~\ref{sec:main}. 
In Section~\ref{sec:universal}, we give a representative network of the ODE-class of the minimal $n$-cell  networks with $n(n-1)$ asymmetric inputs which is the union of $n(n-1)$ feed-forward networks. 
In Section~\ref{sec:lower_min_1_input}, we propose two algorithms to describe minimal $n$-cell networks with one asymmetric input of distinct ODE-classes. We end with some final conclusions in Section~\ref{sec:concl} where, in particular, we present two conjectures about the number of minimal networks.

\section{Preliminary definitions}\label{sec:pre}

In this section, we recall a few definitions and results concerning coupled cell networks, coupled cell systems and ODE-equivalence of networks. 
We follow the coupled cell network formalism of Stewart, Golubitsky and Pivato~\cite{SGP03} and  Golubistky, Stewart and T\"{o}r\"{o}k~\cite{GST05}. 
 
\begin{Def} \normalfont
A {\it (coupled cell) network} $G$ consists of a finite nonempty set $C$ of {\it cells} and a finite nonempty set $E= \{ (c,d):\ c,d \in C\}$ of {\it edges}. Each pair $(c,d)\in E$ represents an edge from cell $d$ to cell $c$ and the cells $c, d$ are called, respectively, the {\it head} and  {\it tail} cell. 
Cells and edges can be of different types. 
\end{Def}

A network can be represented by a directed unweighted graph, where the nodes represent the cells and the edges are depicted by directed arrows. Different types of cells and edges are indicated in the graph, respectively, by different shapes of nodes and different edge arrowheads.

\begin{Def} \normalfont \label{definputset} 
A {\it network with $k$ asymmetric inputs}  is a network with $k$ edge types where each cell receives exactly one edge of each type.
\end{Def} 

Every network with $k$ asymmetric inputs is a homogenous network. 

\begin{Def} \normalfont \label{defhomognet} 
A network is said to be {\it homogeneous} if the cells have all the same type, that is, they are  {\it identical}, 
and receive the same number of input edges per edge type.
\end{Def}

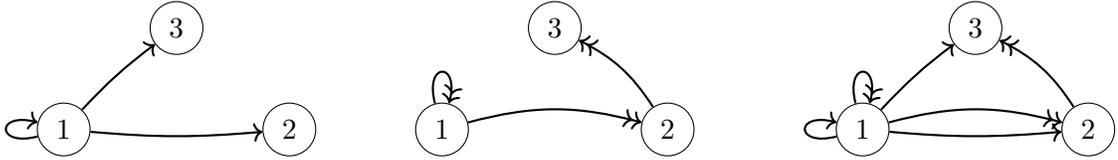
\begin{figure}
\begin{center}
\begin{tabular}{ccc}
\begin{tikzpicture}
[scale=.15,auto=left, node distance=1.5cm, every node/.style={circle,draw}]
 \node[fill=white] (n1) at (4,0) {\small{1}};
  \node[fill=white] (n2) at (24,0) {\small{2}};
 \node[fill=white] (n3) at (14,9)  {\small{3}};

\draw[ arrows={->}, thick]  (n1) to [loop left] (n1);
\draw[arrows={->}, thick] (n1) edge  [bend left=-5] (n2);
\draw[arrows={->}, thick]  (n1) edge  [bend left=5]  (n3);
 \end{tikzpicture}  \qquad  &  \qquad 
 \begin{tikzpicture}
 [scale=.15,auto=left, node distance=1.5cm, every node/.style={circle,draw}]
 \node[fill=white] (n1) at (4,0) {\small{1}};
  \node[fill=white] (n2) at (24,0) {\small{2}};
 \node[fill=white] (n3) at (14,9)  {\small{3}};

 \draw[arrows={->>}, thick] (n1) to [loop above] (n1);
\draw[arrows={->>}, thick]  (n1) edge  [bend left=15] (n2);
\draw[arrows={->>}, thick]  (n2) edge  [bend left=-15] (n3);
 
 \end{tikzpicture}  \qquad  &  \qquad 
 \begin{tikzpicture}
[scale=.15,auto=left, node distance=1.5cm, every node/.style={circle,draw}]
 \node[fill=white] (n1) at (4,0) {\small{1}};
  \node[fill=white] (n2) at (24,0) {\small{2}};
 \node[fill=white] (n3) at (14,9)  {\small{3}};

\draw[ arrows={->}, thick]  (n1) to[loop left] (n1);
\draw[arrows={->}, thick] (n1) edge  [bend left=-5] (n2);
\draw[arrows={->}, thick]  (n1) edge  [bend left=5]  (n3);

\draw[arrows={->>}, thick] (n1) to [loop above] (n1);
\draw[arrows={->>}, thick]  (n1) edge  [bend left=15] (n2);
\draw[arrows={->>}, thick]  (n2) edge  [bend left=-15] (n3);
\end{tikzpicture}
 \end{tabular}
\caption{Networks with three cells and asymmetric inputs: the left and the middle networks have one input; 
the right network has two asymmetric inputs.}
		\label{net_examples}
\end{center}
\end{figure}

\begin{exam} 
In Figure~\ref{net_examples}, we present three-cell networks with one and two asymmetric inputs. 
\hfill 
$\Diamond$
\end{exam}

\begin{Def} \normalfont  
Given a network with set of cells $C$, we say there is a {\it directed path} connecting a sequence of cells $(c_0,c_1, \ldots ,c_{k-1},c_{k})$ of $C$, if there is an edge from $c_{j-1}$ to $c_{j}$, for $j \in \{1,...,k\}$. 
If, for every $j \in \{1,...,k\}$, there is an edge from $c_{j-1}$ to $c_{j}$ or from $c_{j}$ to $c_{j-1}$, we say that there is an {\it undirected path} connecting the sequence of cells $(c_0,c_1, \ldots ,c_{k-1},c_{k})$.
A network is {\it connected} if there is an undirected path between any two cells.
And a network is {\it strongly connected} if there is a directed path connecting all the cells.
\end{Def}

The coupling structure of a network with set of cells $C=\{c_1, \ldots, c_n\}$ and $k$ edge types can be described through $k$ {\it adjacency matrices} $A_l:=(a_{ij}^{(l)}) \in M_{n,n}(\R)$, with rows and columns indexed by the cells in $C$ and $1\leq l\leq k$. Each  entry $a_{	ij}^{(l)}$ corresponds to the number of edges of type $l$ from cell $c_j$ to cell $c_i$. If the network has asymmetric inputs then its adjacency matrices have  entries  $0$ or $1$.

\begin{exam} 
The three-cell network on the right  in Figure~\ref{net_examples} has two asymmetric inputs. Its coupling structure can be represented by the following two $3 \times 3$  adjacency matrices (corresponding, respectively, to the adjacency matrices of the networks on the left and the middle of Figure~\ref{net_examples}):
$$
A_1 = 
\left(
\begin{array}{lll}
1 & 0 & 0 \\
1 & 0 & 0 \\
1 & 0 & 0 
\end{array}
\right), \qquad 
A_2 = 
\left(
\begin{array}{lll}
1 & 0 & 0 \\
1 & 0 & 0 \\
0 & 1 & 0 
\end{array}
\right)\, .
$$
\hfill 
$\Diamond$
\end{exam}

According to the definition of union of graphs, we have the following definition for the union of two networks with the same set of cells but having different edge-types.
\begin{Def} \label{def:union}
Given $k$ networks $G_i$ with the same set of cells $C$, and sets of edges $E_i$, for $i=1, \ldots, k$, we define the {\it union network} $G_1\cup \cdots \cup G_k$, to be the  network with set of cells $C$ and set of edges $E_1 \cup \cdots \cup E_k$.  The set of adjacency matrices of the union network is the disjoint union of the corresponding sets of adjacency matrices.
\end{Def}

\begin{exam} \normalfont 
A network with $k$ asymmetric inputs is the union of $k$ networks with one (asymmetric) input.  The network on the right of Figure~\ref{net_examples}  is the union of the networks on the middle and the left. 
\hfill 
$\Diamond$
\end{exam}

Among the networks with one (asymmetric) input some are {\it feed-forward} networks: they have one cell with a self-loop and tails with root at that cell. A {\it tail} with length $n$  is a directed path connecting a sequence of $n+1$ cells from a root cell to a cell with no outgoing connections.

\begin{exam} \normalfont 
The networks on the left and middle of Figure~\ref{net_examples}  are feed-forward with one input. The network on the left has two tails with length one and the network on the middle has one tail with length two.
\hfill 
$\Diamond$
\end{exam}

\subsection{Coupled cell systems}

Let $G$ be an $n$-cell network with $k$ asymmetric inputs, say of types $1, \ldots, k$. Following~\cite{SGP03, GST05}, we take a cell to be a system of ordinary differential equations and we consider the class of coupled cell systems that have structure consistent with the network $G$. All the cells have the same phase space, say $V = \R^m$ for some $m >0$, the same  internal dynamics and, for each cell $i$, the dynamics is governed by the same smooth function $f$, evaluated at the starting cells of the edges targeting that cell. Thus, for $i=1, \ldots, n$, we have that the evolution of cell $i$ is given by the set of ordinary differential equations 
\begin{equation} \label{eq:CCS1}
\dot{x}_i = f\left(x_i; \, x_{i_1}, \ldots,  x_{i_k}\right), 
\end{equation}
if the input set of cell $i$ is  $\{i_1, \ldots, i_k\}$, where $i_j$ is the tail cell of the edge with type $j$ and head cell $i$. The function $f:\, V^{k+1} \to V$ is assumed to be smooth. We say that coupled cell systems with cells governed by equations of the form (\ref{eq:CCS1}) are $G$-{\it admissible}.

\begin{exam} \normalfont 
Consider the networks on the left and the right of Figure~\ref{net_examples}. Coupled cell systems with structure consistent with these, have the following form, respectively: 
$$
\left\{ 
\begin{array}{l}
\dot{x}_1 = f(x_1; x_1)\\
\dot{x}_2 = f(x_2; x_1)\\
\dot{x}_3 = f(x_3; x_1)
\end{array} 
\right.
\qquad \qquad 
\left\{ 
\begin{array}{l}
\dot{x}_1 = g(x_1; x_1;x_1)\\
\dot{x}_2 = g(x_2; x_1;x_1)\\
\dot{x}_3 = g(x_3; x_1;x_2)
\end{array} 
\right.
$$
for any smooth functions $f:\, \left( \R^m\right)^2 \to  \R^m$ and $g:\,  \left(\R^m\right)^3 \to  \R^m$, if cell phase spaces are chosen to be  $\R^m$. 
\hfill 
$\Diamond$
\end{exam}

\subsection{Network synchrony subspaces}

A network {\it synchrony subspace} $\Delta$  is  a subspace of the network total phase space defined by certain equalities of cell coordinates (a polydiagonal subspace) which is left invariant under the flow of every network admissible coupled cell system. In that case, if $x_i = x_j$ is  one of the cell coordinates defining $\Delta$, then a solution of any system given by (\ref{eq:CCS1}) with initial condition in $\Delta$ have cells $i,j$ synchronized (i.e., $x_i(t) = x_j(t)$) for all time $t$.  One of the consequences of Theorem 6.5 of \cite{SGP03} is that a polydiagonal space $\Delta$ is a synchrony subspace if and only if it is left invariant  under the network adjacency matrices.
So, a polydiagonal space is a synchrony subspace for a union network if and only if it is a synchrony subspace for each network.

\begin{exam} \normalfont Consider the networks  of Figure~\ref{net_examples}. The diagonal space defined by $x_1 = x_2 = x_3$ is a synchrony subspace for the three networks. In fact, any polydiagonal is a synchrony subspace for the network on the left and the subspace defined by $x_1=x_2$ is a synchrony subspace for the middle network. Thus that subspace, $x_1=x_2$, is a synchrony subspace for the network in the right. 
\hfill 
$\Diamond$
\end{exam}

\section{ODE-equivalence of networks}\label{sec:ODE} 
It was noted in ~\cite{SGP03} that different networks with the same number of cells can have the same set of admissible equations for any choice of cell phase spaces. 
As an example of that, consider the two networks in Figure~\ref{silly}. Note that the corresponding coupled cell systems with structure consistent with these, have the following form, respectively: 
$$
\left\{ 
\begin{array}{l}
\dot{x}_1 = f(x_1; x_1)\\
\dot{x}_2 = f(x_2; x_1)\\
\dot{x}_3 = f(x_3; x_1)
\end{array} 
\right.
\qquad \qquad 
\left\{ 
\begin{array}{l}
\dot{x}_1 = g(x_1; x_1;x_1)\\
\dot{x}_2 = g(x_2; x_1;x_1)\\
\dot{x}_3 = g(x_3; x_1;x_1)
\end{array} 
\right.
$$
for any smooth functions $f:\, \left( \R^m\right)^2 \to  \R^m$ and $g:\,  \left(\R^m\right)^3 \to  \R^m$, if cell phase spaces are chosen to be  $\R^m$.  Trivially, given $f$ we can define $g$ in the following form: $g(x,y,z) = f(x,y)$. Also, given $g$, we can define $f$ such that $f(x,y) = g(x,y,y)$. Thus, we have two networks where the associated sets of vector fields coincide.

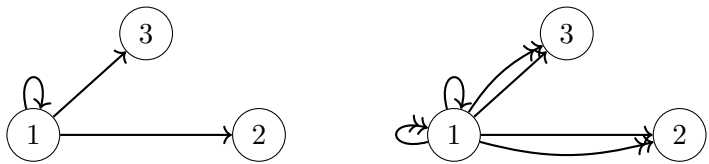
\begin{figure}
\begin{center}
\begin{tabular}{cc}
\begin{tikzpicture}
 [scale=.15,auto=left, node distance=1.5cm, every node/.style={circle,draw}]
 \node[fill=white] (n1) at (4,0) {\small{1}};
  \node[fill=white] (n2) at (24,0) {\small{2}};
 \node[fill=white] (n3) at (14,9)  {\small{3}};
\draw[->, thick] (n1) to[loop above] (n1); 
 \draw[->, thick] (n1) edge (n2); 
 \draw[->, thick] (n1) edge   (n3); 
\draw[->>,white, thick] (n1) to [loop left] (n1); 

 \end{tikzpicture}  \qquad & \qquad 
 \begin{tikzpicture}
[scale=.15,auto=left, node distance=1.5cm, every node/.style={circle,draw}]
 \node[fill=white] (n1) at (4,0) {\small{1}};
  \node[fill=white] (n2) at (24,0) {\small{2}};
 \node[fill=white] (n3) at (14,9)  {\small{3}};

\draw[ arrows={->}, thick]  (n1) to[loop above] (n1);
\draw[arrows={->}, thick] (n1) edge  (n2);
\draw[arrows={->}, thick]  (n1) edge   (n3);

\draw[->>, thick] (n1) to [loop left] (n1); 
 \draw[->>, thick] (n1) edge  [bend right=15] (n2); 
 \draw[->>, thick] (n1) edge    [bend right=-15] (n3); 
\end{tikzpicture}
 \end{tabular}
\caption{Two networks with three cells and asymmetric inputs that are ODE-equivalent.}
		\label{silly}
\end{center}
\end{figure}

   The next definition corresponds to  Definitions 5.1 and 6.2 in ~\cite{DS05}. There is also the more combinatorial approach presented by Agarwal and Field~\cite{AF10a, AF10b}.

\begin{Def}{\cite{DS05}} \normalfont  
Two $n$-cell networks $G_1$ and $G_2$  are {\it ODE-equivalent} when there is a bijection map between their sets of cells such that, for any choice of their cells phase spaces preserving this bijection between the sets of cells, they define the same set of admissible coupled cell systems. If this holds for the set  of linear admissible coupled cell systems, then $G_1$ and $G_2$  are said to be {\it linearly equivalent}.
\end{Def}

The following theorem which corresponds to  Theorem 7.1 of \cite{DS05} relates the two concepts of ODE-equivalence and linear equivalence on networks: 

\begin{thm}{\cite{DS05}} \label{thm:minimal}
Two $n$-cell networks $G_1$ and $G_2$  are ODE-equivalent if and only if they are linearly equivalent.
\end{thm}

In fact, by Corollary 7.9 in \cite{DS05}, we have that two $n$-cell networks, $G_1$ and $G_2$, are 
ODE-equivalent if and only if it is possible to identify through a bijection the corresponding sets of cells such that 
there is equality between the two linear subspaces of  $M_{n \times n} (\R)$ generated by $\mbox{Id}_n, A_1, \dots, A_{k_1}$ and $\mbox{Id}_n, B_1, \dots, B_{k_2}$, where $A_1, \dots, A_{k_1}$ and $B_1, \dots, B_{k_2}$ are the adjacency matrices of $G_1$ and $G_2$, respectively. 

\begin{exam}
In Figure~\ref{silly}, note that the network on the right has two edge types represented by the same adjacency matrix. 
Trivially, using the linear equivalence criterion, the two networks in Figure~\ref{silly} are ODE-equivalent. 
\hfill
$\Diamond$
\end{exam}

\section{Criterion for minimality of networks with asymmetric inputs}\label{sec:criterion}

Fixing the number $n$ of cells, and given an $n$-cell network $G$, the ODE-class of $G$, denoted by $[G]$, is the set of all $n$-cell networks that are ODE-equivalent to $[G]$, which is in general non-finite. In Aguiar and Dias~\cite{AD07}, it was  introduced the notion of {\it minimal networks} of an ODE-class of a network $G$, which are the  networks with the minimal number of edges among the set $[G]$ of all the networks that are ODE-equivalent to $G$.

  \begin{exam}
As noted above, the two networks in Figure~\ref{silly} are ODE-equivalent. We see that each cell in the network on the left receives a unique input. It follows that this network is minimal. In fact, from Proposition 5.11 of Aguiar and Dias~\cite{AD07}, we have that, up to permutation of the cells, the network on the left is the unique minimal network in the ODE-class of both networks of Figure~\ref{silly}. 
\hfill
$\Diamond$
\end{exam}

      In \cite{AD07}, it was also observed that, in general, fixing a network  ODE-class, there are several networks which are minimal. Moreover, it was obtained a method to describe all the minimal networks of the class - that method, is precisely obtained making use of Theorem~\ref{thm:minimal}. We are interested in networks with asymmetric inputs that are minimal. The next result follows from Proposition 7.11 in \cite{AD07}.

\begin{prop} \label{prop:min}
Let $G$ be an $n$-cell network with $m$ asymmetric inputs where $A_1, \ldots, A_m$ are the associated  adjacency matrices. The network $G$ is minimal if and only if the $m+1$ matrices $Id_n, A_1, \ldots, A_m$ are linearly independent.
\end{prop}

\subsection{Minimal n-cell networks with one (asymmetric) input}

Consider that $G$ is an  $n$-cell network with one (asymmetric) input and adjacency matrix $A$ such that $A \neq {\mathrm Id}_n$.  Trivially, we have that ${\mathrm Id}_n$ and $A$ are linearly independent.
Thus a direct consequence of Proposition~\ref{prop:min} is that $G$ is minimal.

Let $\mbox{Min}_{m,n}$ denote the set of minimal $n$-cell networks with $m$ asymmetric inputs. In particular, $\mbox{Min}_{1,n}$ denotes the set of minimal $n$-cell networks with one (asymmetric) input. 

For $n \in \NN$, the number of networks in $\mbox{Min}_{1,n}$, up to permutation of cells,  it is given by Theorem 8.3 in ~\cite{AS05} with $r=1$. See the first column of Table 2 in ~\cite{AS05} for $n \le 6$. The number of connected networks in $\mbox{Min}_{1,n}$, up to permutation of cells,  it is given by Theorem 8.10 in ~\cite{AS05} with $r=1$. See the first column of Table 3 in ~\cite{AS05} for $n \le 6$. 
Note that, as stated in \cite{AS05},  these $n$-cell networks with one input are in one-to-one correspondence with the distinct mappings of $n$ points to themselves: given such a map $f$, we can take the $n$-cell network where each cell $i$ receives an input edge from cell $f(i)$. 
Moreover, up to re-enumeration of the cells, they are not ODE-equivalent: 
\begin{prop} \label{prop:ODE_min_one}
Let $G_1$ and $G_2$ be two minimal $n$-cell networks with one (asymmetric) input  and adjacency matrices $A_i \ne \mbox{Id}_n$,  for $i=1,2$. 
Then $[G_1] = [G_2]$ if and only if $G_1$ and $G_2$ are equal up to permutation of cells. 
Equivalently, $[G_1] = [G_2]$ if and only if it exists an $n \times n$ permutation matrix $P$ such that $A_1 = P A_2 P^{-1}$. 
\end{prop}

Note that the statement of Proposition~\ref{prop:ODE_min_one} can also be derived from  Proposition 5.11 of \cite{AD07}  where it is proved that if $G$ is a network with one asymmetric input then $G$ is the unique minimal network of the class $[G]$,  up to re-enumeration of the cells.

We have, then, the following result.
\begin{thm}	
Let $n$ be a positive integer. The number of distinct ODE-classes at the set $\mbox{Min}_{1,n}$ is given by Theorem 8.3 in ~\cite{AS05} with $r=1$. The number of distinct ODE-classes of connected networks at the set $\mbox{Min}_{1,n}$ is given by Theorem 8.10 in ~\cite{AS05} with $r=1$.
\end{thm}

\subsection{Minimal n-cell networks with two asymmetric inputs}

For the particular case of a network $G$ with two asymmetric inputs, the result in Proposition~\ref{prop:min} states that $G$ is minimal if and only if the adjacency matrices $A_1$ and $A_2$ of $G$ and the identity matrix (of the same dimension) are linearly independent. We get then  the following corollary of Proposition~\ref{prop:min}: 

\begin{cor}\label{cor:notmin}
A network $G$ with two asymmetric inputs given by the valency one adjacency matrices $A_i\not=id_n$, for $i=1,2$, where $A_1 \not= A_2$ is minimal. 
\end{cor}

\begin{proof}
By Proposition~\ref{prop:min}, $G$ is not minimal  if and only if the matrices $\mbox{Id}_n,\, A_1,\, A_2$ are linearly dependent. As the matrices $A_1$ and $A_2$ have valency one and are not the identity matrix, then $\mbox{Id}_n, A_1$ are linearly independent and $\mbox{Id}_n, A_2$ are linearly independent. Thus if $\mbox{Id}_n, \, A_1,\,  A_2$ are linearly dependent, then there are nonzero real entries $a, b, c$ such that
$$
a \mbox{Id}_n + b A_1 + c A_2 = 0_{m\times n}\, .
$$
Without loss of generality, we assume that $A_2$ is a linear combination of $\mbox{Id}_n$ and $A_1$.
Thus, there are real numbers $\alpha$ and $\beta$ such that 
$$
A_2 = \alpha \mbox{Id}_n + \beta A_1\, .
$$
As $A_1 \not= \mbox{Id}_n$, the matrices $A_1$ and $\mbox{Id}_n$ have at least one row $i$ such that two entries differ and so, we can find  $j$ with $j\not=i$ such that $(A_1)_{ij}=1$ and $(A_1)_{ii}=0$. We obtain two linear equations: taking $k_1 = (A_2)_{ij}$ and $k_2 = (A_2)_{ii}$, 
$$
\left\{ 
\begin{array}{l}
(A_2)_{ij} = \alpha (\mbox{Id}_n)_{ij} + \beta (A_1)_{ij}\\
(A_2)_{ii} = \alpha (\mbox{Id}_n)_{ii} + \beta (A_1)_{ii}
\end{array}
\right. \quad \Leftrightarrow \quad 
\left\{ 
\begin{array}{l}
0 \alpha + 1 \beta = k_1\\
1 \alpha + 0 \beta = k_2
\end{array}
\right. \, .
$$
Thus $\beta = k_1\in \{0,1\}$ and $\alpha = k_2\in \{0,1\}$. Therefore we have one of the following cases $A_2=Id_n+A_1$, $A_2=A_1$, $A_2=Id_n$ or $A_2=0$. By assumption, all those cases are impossible.
Thus $\mbox{Id}_n,\, A_1,\, A_2$ are linearly independent and $G$ is minimal.
\end{proof}

It follows from Corollary~\ref{cor:notmin} that  an  $n$-cell network with two asymmetric inputs is not minimal if and only if the two inputs are equal. In this case the network is ODE-equivalent to an  $n$-cell network with one (asymmetric) input.

\subsection{Minimal n-cell networks with k asymmetric inputs}

By Proposition~\ref{prop:min} and Theorem~\ref{thm:minimal}, it also follows that: 

\begin{cor}\label{cor:pminimal}
Let $G$ be an $n$-cell network with $k$ asymmetric inputs and adjacency matrices $A_1, \, \ldots,\, A_k$. If $p$ denotes the dimension of the linear space generated by $\mbox{Id}_n$ and $A_1, \, \ldots,\, A_k$, then $G$ is ODE-equivalent to a minimal $n$-cell network with $p-1$ asymmetric inputs. 
\end{cor}

\begin{rem}\normalfont  
Under the conditions of Corollary~\ref{cor:pminimal}, any set of $p-1$ adjacency matrices of $G$, say $A_1,\, \ldots,\, A_{p-1}$, such that $\mbox{Id}_n,\ A_1, \ldots,\, A_{p-1}$ are linearly independent, define a minimal network with $p-1$ asymmetric inputs in the ODE-class $[G]$. 
\hfill 
$\Box$
\end{rem}

\section{Classification of three-cell networks with two asymmetric inputs}\label{sec:casestudy}

Using the fact that a network with $k$ asymmetric inputs is the union of $k$ networks with one input, we have a way of enumerating network with $k$ asymmetric inputs using the enumeration of networks with one input. 
This list is large and the concept of minimality and ODE-equivalence of networks can be used to restrict this list. We illustrate this method with networks with three cells and two asymmetric inputs. That is, we obtain all the minimal three-cell connected networks with two asymmetric inputs, up to ODE-equivalence. 

We start by classifying the three-cell minimal networks with one (asymmetric) input. 

\subsection{Classification of three-cell networks with one (asymmetric) input}\label{subsec:m31}

We state and prove a well known classification of the ODE-classes of the minimal three-cell networks with one (asymmetric) input. 
See, for example, Leite and Golubitsky~\cite{LG06}. 
We include this classification for completeness as it will be used in the next sections. 
We also include the two-dimension synchrony subspaces of those minimal representative networks. 

\begin{lemma}\label{lem:um3}
There are only seven ODE-classes of three-cell networks with one input. One of these classes corresponds to the disconnected three cell network with adjacency matrix $\mbox{Id}_3$. The other six classes are represented by the six minimal networks in Table~\ref{tab:val13cell}. 
\end{lemma}

\begin{proof} Let $G$ be a three-cell network with one (asymmetric) input and adjacency matrix $A \not= \mbox{Id}_3$. \\
\noindent  (i) If every cell of $G$  sends some input then:  either $G$ is  the  3-cycle and it has no two-dimensional synchrony subspaces, 
see network $A$ of Table~\ref{tab:val13cell}; or $G$ has a cell $i$ with a self-loop and a 2-cycle and it has exactly one two-dimensional synchrony subspace,  $\Delta_i = \{ x:\, x_j = x_k \mbox{ where } j,k \not=i\}$, see network $B$ of Table~\ref{tab:val13cell}. Moreover, there are no more two-dimensional synchrony subspace since cell $i$ cannot synchronize with only one of the two other cells.  \\
\noindent (ii)  If two cells of $G$ do not send any input to the other cells, then the third cell  has to send all the three edges including a self-loop and $G$ has three two-dimensional synchrony subspaces. 
 Equivalently, every two cells can synchronize. See network $C$ of Table~\ref{tab:val13cell}. \\
\noindent (iii) If exactly one cell of $G$ does not send any input to the other cells, then it must receive an edge from a second cell. 
If this second cell does not send another edge, then the third cell must send two edges including a self-loop.
Thus, in this case $G$ is the network $D$ of Table~\ref{tab:val13cell} and has exactly one two-dimensional synchrony subspace.
If the second cell sends another edge, then the second and third cell must send each an edge between them. In this case, they can send self-loops corresponding to network $E$ of Table~\ref{tab:val13cell} or form a $2$-cycle corresponding to network $F$ of Table~\ref{tab:val13cell}. Moreover, the networks $E$ and $F$ have exactly two two-dimensional synchrony subspaces. 
\end{proof}

\begin{table}
 \begin{center}
 \resizebox{1 \textwidth}{!}{ 
 {\tiny 
 \begin{tabular}{|cc|c|c||cc|c|c|}
\hline 
% & & & & \\
&  & 2D   & Adjacency & &  & 2D   & Adjacency  \\
& Network &  Synchrony &  Matrix & & Network &  Synchrony &  Matrix  \\
&  &  Subspaces &     & &  &  Subspaces &    \\

\hline 
A &
\begin{tikzpicture}
 [scale=.15,auto=left, node distance=1.5cm, every node/.style={circle,draw}]
 \node[fill=white] (n1) at (4,0) {\small{1}};
  \node[fill=white] (n2) at (24,0) {\small{2}};
 \node[fill=white] (n3) at (14,9)  {\small{3}};
 \draw[->, thick] (n1) edge  [bend left=-10] (n2); 
 \draw[->, thick] (n2) edge  [bend left=-10] (n3); 
\draw[->, thick] (n3) edge [bend right=10] (n1); 
\end{tikzpicture} & 
- & 
$\left[
\begin{array}{ccc}
0 & 0 & 1 \\
1 & 0 & 0 \\
0 & 1 & 0
\end{array}
\right]$ & 
%$ \begin{array}{rl}
%1 & (1,1,1) \\
%\psi & (1, \psi^2, \psi) \\
% \psi^2 & (1, \psi, \psi^2) 
% \end{array}$ 
% &
B & 
\begin{tikzpicture}
 [scale=.15,auto=left, node distance=1.5cm, every node/.style={circle,draw}]
 \node[fill=white] (n1) at (4,0) {\small{1}};
  \node[fill=white] (n2) at (24,0) {\small{2}};
 \node[fill=white] (n3) at (14,9)  {\small{3}};
 \draw[->, thick] (n1) to [loop above] (n1); 
 \draw[->, thick] (n2) edge  [bend left=-10] (n3); 
\draw[->, thick] (n3) edge [bend right=10] (n2); 
\end{tikzpicture} & 
$\Delta_1 $&%= \{ x:\, x_2 = x_3\}$ & 
$\left[
\begin{array}{ccc}
1 & 0 & 0 \\
0 & 0 & 1 \\
0 & 1 & 0
\end{array}
\right]$ 
%& 
%$\begin{array}{rl}
%1 & (1,1,1)\\
%1 & (1,0,0)\\
%-1 & (0,1,-1) 
%\end{array}$
\\
\hline 
C& 
\begin{tikzpicture}
 [scale=.15,auto=left, node distance=1.5cm, every node/.style={circle,draw}]
 \node[fill=white] (n1) at (4,0) {\small{1}};
  \node[fill=white] (n2) at (24,0) {\small{2}};
 \node[fill=white] (n3) at (14,9)  {\small{3}};
 \draw[->, thick] (n1) to [loop above] (n1); 
 \draw[->, thick] (n1) edge  [bend left=-10] (n2); 
 \draw[->, thick] (n1) edge  [bend left=-10] (n3); 
 \end{tikzpicture} & 
 $\begin{array}{l}
\Delta_1 \\%= \{ x:\, x_2 = x_3\}\\
\Delta_2 \\%= \{ x:\, x_1 = x_3\}\\
\Delta_3 \\%= \{ x:\, x_1 = x_2\}
\end{array}$ &
$\left[ 
\begin{array}{ccc}
1 & 0 & 0 \\
1 & 0 & 0 \\
1 & 0 & 0
\end{array}
\right]$ & 
%$\begin{array}{rl}
%1 & (1,1,1)\\
%0 & (0,1,0)\\
%0 & (0,0,1) 
%\end{array}$
%&
D & 
\begin{tikzpicture}
 [scale=.15,auto=left, node distance=1.5cm, every node/.style={circle,draw}]
 \node[fill=white] (n1) at (4,0) {\small{1}};
  \node[fill=white] (n2) at (24,0) {\small{2}};
 \node[fill=white] (n3) at (14,9)  {\small{3}};
 \draw[->, thick] (n1) to [loop above] (n1); 
 \draw[->, thick] (n1) edge  [bend left=-10] (n2); 
 \draw[->, thick] (n2) edge  [bend left=-10] (n3); 
 \end{tikzpicture} &
 $\begin{array}{l}
\Delta_3 %= \{ x:\, x_1 = x_2\}
\end{array}$ &
$\left[
\begin{array}{ccc}
1 & 0 & 0 \\
1 & 0 & 0 \\
0 & 1 & 0
\end{array}
\right]$ 
%& 
%$\begin{array}{rl}
%1 & (1,1,1)\\
%0^* & (0,0,1)  
%\end{array}$
\\
\hline 
E & 
\begin{tikzpicture}
 [scale=.15,auto=left, node distance=1.5cm, every node/.style={circle,draw}]
 \node[fill=white] (n1) at (4,0) {\small{1}};
  \node[fill=white] (n2) at (24,0) {\small{2}};
 \node[fill=white] (n3) at (14,9)  {\small{3}};
 \draw[->, thick] (n1) to [loop above] (n1); 
 \draw[->, thick] (n2) to [loop above] (n2); 
\draw[->, thick] (n1) edge [bend right=10] (n3); 
\end{tikzpicture} &
$\begin{array}{l}
\Delta_2 \\%= \{ x:\, x_1 = x_3\}\\
\Delta_3 \\%= \{ x:\, x_1 = x_2\}
\end{array}$ &
$\left[ 
\begin{array}{ccc}
1 & 0 & 0 \\
0 & 1 & 0 \\
1 & 0 & 0
\end{array}\right]$ 
%& 
%$\begin{array}{rl}
%1 & (1,1,1)\\
%1 & (0,1,0)\\
%0 & (0,0,1) 
%\end{array}$
&
F & 
\begin{tikzpicture}
 [scale=.15,auto=left, node distance=1.5cm, every node/.style={circle,draw}]
 \node[fill=white] (n1) at (4,0) {\small{1}};
  \node[fill=white] (n2) at (24,0) {\small{2}};
 \node[fill=white] (n3) at (14,9)  {\small{3}};
 \draw[->, thick] (n1) edge  [bend left=-10] (n2); 
 \draw[->, thick] (n2) edge  [bend left=-10] (n1); 
\draw[->, thick] (n1) edge [bend right=10] (n3); 
\end{tikzpicture} &
$\begin{array}{l}
\Delta_1 \\%= \{ x:\, x_2 = x_3\}\\
\Delta_3 \\%= \{ x:\, x_1 = x_2\}
\end{array}$ &
$\left[
\begin{array}{ccc}
0 & 1 & 0 \\
1 & 0 & 0 \\
1 & 0 & 0
\end{array}\right]$ 
%& 
%$\begin{array}{rl}
%1 & (1,1,1)\\
%-1 & (-1,1,1)\\
%0 & (0,0,1) 
%\end{array}$
\\
\hline 
\end{tabular}}
}
\end{center}
\caption{Three-cell networks with one (asymmetric) input and adjacency matrix $A \not= \mbox{Id}_3$, up to re-enumeration of the cells.  
Note that the networks $C$ and $D$ are feed-forward. 
} \label{tab:val13cell}
\end{table}

\subsection{Classification of three-cell networks with two asymmetric inputs}\label{sec:3cell2inputs}

\begin{table}
\resizebox*{!}{0.93\textheight}{
  {\tiny 
\begin{tabular}{|c|c||c|c||c|c|}
 \hline 
% & & &  &  & \\
Network & 2D Syn  & Network & 2D Syn  &  Network & 2D Syn   \\
              & Subspace         &               & Subspace   &            & Subspace     \\
\hline 
\hline 
$A_1$
\begin{tikzpicture}
 [scale=.15,auto=left, node distance=1.5cm, every node/.style={circle,draw}]
 \node[fill=white] (n1) at (4,0) {\small{1}};
  \node[fill=white] (n2) at (24,0) {\small{2}};
 \node[fill=white] (n3) at (14,9)  {\small{3}};
 \draw[->, thick] (n1) edge  [bend left=-10] (n2); 
 \draw[->, thick] (n2) edge  [bend left=-10] (n3); 
\draw[->, thick] (n3) edge [bend right=10] (n1); 
\end{tikzpicture} & 
- & 
$A_2$ 
\begin{tikzpicture}
 [scale=.15,auto=left, node distance=1.5cm, every node/.style={circle,draw}]
 \node[fill=white] (n1) at (4,0) {\small{1}};
  \node[fill=white] (n2) at (24,0) {\small{2}};
 \node[fill=white] (n3) at (14,9)  {\small{3}};
 \draw[->, thick] (n1) edge  [bend left=10] (n3); 
 \draw[->, thick] (n3) edge  [bend left=10] (n2); 
\draw[->, thick] (n2) edge [bend right=-10] (n1); 
\end{tikzpicture} & 
- & &  \\
\hline 
\hline 
$B_1$ 
\begin{tikzpicture}
 [scale=.15,auto=left, node distance=1.5cm, every node/.style={circle,draw}]
 \node[fill=white] (n1) at (4,0) {\small{1}};
  \node[fill=white] (n2) at (24,0) {\small{2}};
 \node[fill=white] (n3) at (14,9)  {\small{3}};
 \draw[->, thick] (n1) to [loop above] (n1); 
 \draw[->, thick] (n2) edge  [bend left=-10] (n3); 
\draw[->, thick] (n3) edge [bend right=10] (n2); 
\end{tikzpicture} & 
$\Delta_1 $ 
&
%= \{ x:\, x_2 = x_3\}$ & 
$B_2$ 
\begin{tikzpicture}
 [scale=.15,auto=left, node distance=1.5cm, every node/.style={circle,draw}]
 \node[fill=white] (n1) at (4,0) {\small{1}};
  \node[fill=white] (n2) at (24,0) {\small{2}};
 \node[fill=white] (n3) at (14,9)  {\small{3}};
 \draw[->, thick] (n2) to [loop above] (n2); 
 \draw[->, thick] (n1) edge  [bend left=-10] (n3); 
\draw[->, thick] (n3) edge [bend right=10] (n1); 
\end{tikzpicture} & 
$\Delta_2 $
%= \{ x:\, x_1 = x_3\}$ \\
&
$B_3$ 
\begin{tikzpicture}
 [scale=.15,auto=left, node distance=1.5cm, every node/.style={circle,draw}]
 \node[fill=white] (n1) at (4,0) {\small{1}};
  \node[fill=white] (n2) at (24,0) {\small{2}};
 \node[fill=white] (n3) at (14,9)  {\small{3}};
 \draw[->, thick] (n3) to [loop above] (n3); 
 \draw[->, thick] (n2) edge  [bend left=-10] (n1); 
\draw[->, thick] (n1) edge [bend right=10] (n2); 
\end{tikzpicture} & 
$\Delta_3 $\\
%= \{ x:\, x_1 = x_2\}$ \\
\hline 
\hline 
$C_1$ 
\begin{tikzpicture}
 [scale=.15,auto=left, node distance=1.5cm, every node/.style={circle,draw}]
 \node[fill=white] (n1) at (4,0) {\small{1}};
  \node[fill=white] (n2) at (24,0) {\small{2}};
 \node[fill=white] (n3) at (14,9)  {\small{3}};
 \draw[->, thick] (n1) to [loop above] (n1); 
 \draw[->, thick] (n1) edge  [bend left=-10] (n2); 
 \draw[->, thick] (n1) edge  [bend left=10] (n3); 
 \end{tikzpicture} & 
 $\begin{array}{l}
\Delta_1 \\%= \{ x:\, x_2 = x_3\}\\
\Delta_2 \\%= \{ x:\, x_1 = x_3\}\\
\Delta_3 %= \{ x:\, x_1 = x_2\}
\end{array}$ &
$C_2$ 
\begin{tikzpicture}
 [scale=.15,auto=left, node distance=1.5cm, every node/.style={circle,draw}]
 \node[fill=white] (n1) at (4,0) {\small{1}};
  \node[fill=white] (n2) at (24,0) {\small{2}};
 \node[fill=white] (n3) at (14,9)  {\small{3}};
 \draw[->, thick] (n2) to [loop above] (n2); 
 \draw[->, thick] (n2) edge  [bend left=10] (n1); 
 \draw[->, thick] (n2) edge  [bend left=-10] (n3); 
 \end{tikzpicture} & 
 $\begin{array}{l}
\Delta_1 \\%= \{ x:\, x_2 = x_3\}\\
\Delta_2 \\%= \{ x:\, x_1 = x_3\}\\
\Delta_3 %= \{ x:\, x_1 = x_2\}
\end{array}$ 
&  
$C_3$ 
\begin{tikzpicture}
 [scale=.15,auto=left, node distance=1.5cm, every node/.style={circle,draw}]
 \node[fill=white] (n1) at (4,0) {\small{1}};
  \node[fill=white] (n2) at (24,0) {\small{2}};
 \node[fill=white] (n3) at (14,9)  {\small{3}};
 \draw[->, thick] (n3) to [loop above] (n3); 
 \draw[->, thick] (n3) edge  [bend left=-10] (n1); 
 \draw[->, thick] (n3) edge  [bend left=10] (n2); 
 \end{tikzpicture} & 
 $\begin{array}{l}
\Delta_1 \\%= \{ x:\, x_2 = x_3\}\\
\Delta_2 \\%= \{ x:\, x_1 = x_3\}\\
\Delta_3 %= \{ x:\, x_1 = x_2\}
\end{array}$ \\
\hline 
\hline 
$D_1$ 
\begin{tikzpicture}
 [scale=.15,auto=left, node distance=1.5cm, every node/.style={circle,draw}]
 \node[fill=white] (n1) at (4,0) {\small{1}};
  \node[fill=white] (n2) at (24,0) {\small{2}};
 \node[fill=white] (n3) at (14,9)  {\small{3}};
 \draw[->, thick] (n1) to [loop above] (n1); 
 \draw[->, thick] (n1) edge  [bend left=-10] (n2); 
 \draw[->, thick] (n2) edge  [bend left=-10] (n3); 
 \end{tikzpicture} &
 $\begin{array}{l}
\Delta_3 %= \{ x:\, x_1 = x_2\}
\end{array}$ &
$D_2$ 
\begin{tikzpicture}
 [scale=.15,auto=left, node distance=1.5cm, every node/.style={circle,draw}]
 \node[fill=white] (n1) at (4,0) {\small{1}};
  \node[fill=white] (n2) at (24,0) {\small{2}};
 \node[fill=white] (n3) at (14,9)  {\small{3}};
 \draw[->, thick] (n1) to [loop above] (n1); 
 \draw[->, thick] (n1) edge  [bend left=10] (n3); 
 \draw[->, thick] (n3) edge  [bend left=10] (n2); 
 \end{tikzpicture} &
 $\begin{array}{l}
\Delta_2 %= \{ x:\, x_1 = x_3\}
\end{array}$ 
& 
$D_3$ 
\begin{tikzpicture}
 [scale=.15,auto=left, node distance=1.5cm, every node/.style={circle,draw}]
 \node[fill=white] (n1) at (4,0) {\small{1}};
  \node[fill=white] (n2) at (24,0) {\small{2}};
 \node[fill=white] (n3) at (14,9)  {\small{3}};
 \draw[->, thick] (n2) to [loop above] (n2); 
 \draw[->, thick] (n2) edge  [bend left=-10] (n3); 
 \draw[->, thick] (n3) edge  [bend left=-10] (n1); 
 \end{tikzpicture} &
 $\begin{array}{l}
\Delta_1 %= \{ x:\, x_2 = x_3\}
\end{array}$ \\
\hline 
$D_4$ 
\begin{tikzpicture}
 [scale=.15,auto=left, node distance=1.5cm, every node/.style={circle,draw}]
 \node[fill=white] (n1) at (4,0) {\small{1}};
  \node[fill=white] (n2) at (24,0) {\small{2}};
 \node[fill=white] (n3) at (14,9)  {\small{3}};
 \draw[->, thick] (n2) to [loop above] (n2); 
 \draw[->, thick] (n2) edge  [bend left=10] (n1); 
 \draw[->, thick] (n1) edge  [bend left=10] (n3); 
 \end{tikzpicture} &
 $\begin{array}{l}
\Delta_3 %= \{ x:\, x_1 = x_2\}
\end{array}$ 
&
$D_5$ 
\begin{tikzpicture}
 [scale=.15,auto=left, node distance=1.5cm, every node/.style={circle,draw}]
 \node[fill=white] (n1) at (4,0) {\small{1}};
  \node[fill=white] (n2) at (24,0) {\small{2}};
 \node[fill=white] (n3) at (14,9)  {\small{3}};
 \draw[->, thick] (n3) to [loop above] (n3); 
 \draw[->, thick] (n3) edge  [bend left=-10] (n1); 
 \draw[->, thick] (n1) edge  [bend left=-10] (n2); 
 \end{tikzpicture} &
 $\begin{array}{l}
\Delta_2 %= \{ x:\, x_1 = x_3\}
\end{array}$ &
$D_6$ 
\begin{tikzpicture}
 [scale=.15,auto=left, node distance=1.5cm, every node/.style={circle,draw}]
 \node[fill=white] (n1) at (4,0) {\small{1}};
  \node[fill=white] (n2) at (24,0) {\small{2}};
 \node[fill=white] (n3) at (14,9)  {\small{3}};
 \draw[->, thick] (n3) to [loop above] (n3); 
 \draw[->, thick] (n3) edge  [bend left=10] (n2); 
 \draw[->, thick] (n2) edge  [bend left=10] (n1); 
 \end{tikzpicture} &
 $\begin{array}{l}
\Delta_1 %= \{ x:\, x_2 = x_3\}
\end{array}$ \\
\hline 
\hline 
$E_1$  
\begin{tikzpicture}
 [scale=.15,auto=left, node distance=1.5cm, every node/.style={circle,draw}]
 \node[fill=white] (n1) at (4,0) {\small{1}};
  \node[fill=white] (n2) at (24,0) {\small{2}};
 \node[fill=white] (n3) at (14,9)  {\small{3}};
 \draw[->, thick] (n1) to [loop above] (n1); 
 \draw[->, thick] (n2) to [loop above] (n2); 
\draw[->, thick] (n1) edge [bend right=-10] (n3); 
\end{tikzpicture} &
$\begin{array}{l}
\Delta_2 \\%= \{ x:\, x_1 = x_3\}\\
\Delta_3 %= \{ x:\, x_1 = x_2\}
\end{array}$ &
$E_2$  
\begin{tikzpicture}
 [scale=.15,auto=left, node distance=1.5cm, every node/.style={circle,draw}]
 \node[fill=white] (n1) at (4,0) {\small{1}};
  \node[fill=white] (n2) at (24,0) {\small{2}};
 \node[fill=white] (n3) at (14,9)  {\small{3}};
 \draw[->, thick] (n3) to [loop above] (n3); 
 \draw[->, thick] (n2) to [loop above] (n2); 
\draw[->, thick] (n3) edge [bend right=10] (n1); 
\end{tikzpicture} &
$\begin{array}{l}
\Delta_1 \\%= \{ x:\, x_2 = x_3\}\\
\Delta_2 %= \{ x:\, x_1 = x_3\}
\end{array}$ 
&
$E_3$  
\begin{tikzpicture}
 [scale=.15,auto=left, node distance=1.5cm, every node/.style={circle,draw}]
 \node[fill=white] (n1) at (4,0) {\small{1}};
  \node[fill=white] (n2) at (24,0) {\small{2}};
 \node[fill=white] (n3) at (14,9)  {\small{3}};
 \draw[->, thick] (n2) to [loop above] (n2); 
 \draw[->, thick] (n1) to [loop above] (n1); 
\draw[->, thick] (n2) edge [bend right=10] (n3); 
\end{tikzpicture} &
$\begin{array}{l}
\Delta_1\\% = \{ x:\, x_2 = x_3\}\\
\Delta_3 %= \{ x:\, x_1 = x_2\}
\end{array}$ \\
\hline  
$E_4$  
\begin{tikzpicture}
 [scale=.15,auto=left, node distance=1.5cm, every node/.style={circle,draw}]
 \node[fill=white] (n1) at (4,0) {\small{1}};
  \node[fill=white] (n2) at (24,0) {\small{2}};
 \node[fill=white] (n3) at (14,9)  {\small{3}};
 \draw[->, thick] (n3) to [loop above] (n3); 
 \draw[->, thick] (n1) to [loop above] (n1); 
\draw[->, thick] (n3) edge [bend right=-10] (n2); 
\end{tikzpicture} &
$\begin{array}{l}
\Delta_1 \\%= \{ x:\, x_2 = x_3\}\\
\Delta_2 %= \{ x:\, x_1 = x_3\}
\end{array}$ 
&
$E_5$  
\begin{tikzpicture}
 [scale=.15,auto=left, node distance=1.5cm, every node/.style={circle,draw}]
 \node[fill=white] (n1) at (4,0) {\small{1}};
  \node[fill=white] (n2) at (24,0) {\small{2}};
 \node[fill=white] (n3) at (14,9)  {\small{3}};
 \draw[->, thick] (n2) to [loop above] (n2); 
 \draw[->, thick] (n3) to [loop above] (n3); 
\draw[->, thick] (n2) edge [bend right=-10] (n1); 
\end{tikzpicture} &
$\begin{array}{l}
\Delta_1 \\%= \{ x:\, x_2 = x_3\}\\
\Delta_3 %= \{ x:\, x_1 = x_2\}
\end{array}$ &
$E_6$  
\begin{tikzpicture}
 [scale=.15,auto=left, node distance=1.5cm, every node/.style={circle,draw}]
 \node[fill=white] (n1) at (4,0) {\small{1}};
  \node[fill=white] (n2) at (24,0) {\small{2}};
 \node[fill=white] (n3) at (14,9)  {\small{3}};
 \draw[->, thick] (n1) to [loop above] (n1); 
 \draw[->, thick] (n3) to [loop above] (n3); 
\draw[->, thick] (n1) edge [bend right=10] (n2); 
\end{tikzpicture} &
$\begin{array}{l}
\Delta_2 \\%= \{ x:\, x_1 = x_3\}\\
\Delta_3 %= \{ x:\, x_1 = x_2\}
\end{array}$ \\
\hline
\hline 
$F_1$ 
\begin{tikzpicture}
 [scale=.15,auto=left, node distance=1.5cm, every node/.style={circle,draw}]
 \node[fill=white] (n1) at (4,0) {\small{1}};
  \node[fill=white] (n2) at (24,0) {\small{2}};
 \node[fill=white] (n3) at (14,9)  {\small{3}};
 \draw[->, thick] (n1) edge  [bend left=-10] (n2); 
 \draw[->, thick] (n2) edge  [bend left=-10] (n1); 
\draw[->, thick] (n1) edge [bend right=-10] (n3); 
\end{tikzpicture} &
$\begin{array}{l}
\Delta_1 \\%= \{ x:\, x_2 = x_3\}\\
\Delta_3 %= \{ x:\, x_1 = x_2\}
\end{array}$ &
$F_2$ 
\begin{tikzpicture}
 [scale=.15,auto=left, node distance=1.5cm, every node/.style={circle,draw}]
 \node[fill=white] (n1) at (4,0) {\small{1}};
  \node[fill=white] (n2) at (24,0) {\small{2}};
 \node[fill=white] (n3) at (14,9)  {\small{3}};
 \draw[->, thick] (n2) edge  [bend left=-10] (n1); 
 \draw[->, thick] (n1) edge  [bend left=-10] (n2); 
\draw[->, thick] (n2) edge [bend right=10] (n3); 
\end{tikzpicture} &
$\begin{array}{l}
\Delta_2\\% = \{ x:\, x_1 = x_3\}\\
\Delta_3 %= \{ x:\, x_1 = x_2\}
\end{array}$
&
$F_3$ 
\begin{tikzpicture}
 [scale=.15,auto=left, node distance=1.5cm, every node/.style={circle,draw}]
 \node[fill=white] (n1) at (4,0) {\small{1}};
  \node[fill=white] (n2) at (24,0) {\small{2}};
 \node[fill=white] (n3) at (14,9)  {\small{3}};
 \draw[->, thick] (n2) edge  [bend left=-10] (n3); 
 \draw[->, thick] (n3) edge  [bend left=-10] (n2); 
\draw[->, thick] (n2) edge [bend right=-10] (n1); 
\end{tikzpicture} &
$\begin{array}{l}
\Delta_1\\% = \{ x:\, x_2 = x_3\}\\
\Delta_2 %= \{ x:\, x_1 = x_3\}
\end{array}$ \\
\hline 
$F_4$ 
\begin{tikzpicture}
 [scale=.15,auto=left, node distance=1.5cm, every node/.style={circle,draw}]
 \node[fill=white] (n1) at (4,0) {\small{1}};
  \node[fill=white] (n2) at (24,0) {\small{2}};
 \node[fill=white] (n3) at (14,9)  {\small{3}};
 \draw[->, thick] (n3) edge  [bend left=-10] (n2); 
 \draw[->, thick] (n2) edge  [bend left=-10] (n3); 
\draw[->, thick] (n3) edge [bend right=10] (n1); 
\end{tikzpicture} &
$\begin{array}{l}
\Delta_1 \\%= \{ x:\, x_2 = x_3\}\\
\Delta_3 %= \{ x:\, x_1 = x_2\}
\end{array}$ 
&
$F_5$ 
\begin{tikzpicture}
 [scale=.15,auto=left, node distance=1.5cm, every node/.style={circle,draw}]
 \node[fill=white] (n1) at (4,0) {\small{1}};
  \node[fill=white] (n2) at (24,0) {\small{2}};
 \node[fill=white] (n3) at (14,9)  {\small{3}};
 \draw[->, thick] (n3) edge  [bend left=-10] (n1); 
 \draw[->, thick] (n1) edge  [bend left=-10] (n3); 
\draw[->, thick] (n3) edge [bend right=-10] (n2); 
\end{tikzpicture} &
$\begin{array}{l}
\Delta_2 \\%= \{ x:\, x_1 = x_3\}\\
\Delta_3 %= \{ x:\, x_1 = x_2\}
\end{array}$ &
$F_6$ 
\begin{tikzpicture}
 [scale=.15,auto=left, node distance=1.5cm, every node/.style={circle,draw}]
 \node[fill=white] (n1) at (4,0) {\small{1}};
  \node[fill=white] (n2) at (24,0) {\small{2}};
 \node[fill=white] (n3) at (14,9)  {\small{3}};
 \draw[->, thick] (n1) edge  [bend left=-10] (n3); 
 \draw[->, thick] (n3) edge  [bend left=-10] (n1); 
\draw[->, thick] (n1) edge [bend right=10] (n2); 
\end{tikzpicture} &
$\begin{array}{l}
\Delta_1\\% = \{ x:\, x_2 = x_3\}\\
\Delta_2 %= \{ x:\, x_1 = x_3\}
\end{array}$\\
\hline 
\end{tabular}}}
\caption{Three-cell networks with one (asymmetric) input obtained from the networks in Table~\ref{tab:val13cell} by permutation of cells. } 
\label{tab:listaval13cell}     
\end{table}

We obtain now all the minimal three-cell connected networks with two asymmetric inputs, up to ODE-equivalence.

As stated before, every three-cell network with two asymmetric inputs is the union of two three-cell networks with one (asymmetric)  input. Since, in the union of two such networks, the order of the cells matters, we list in Table~\ref{tab:listaval13cell} 
all the three-cell networks with one (asymmetric) input and adjacency matrix $A \not= \mbox{Id}_3$, which are obtained from the networks in Table~\ref{tab:val13cell} by permutation of the three cells.

By Corollary~\ref{cor:notmin},  a  three-cell network with two asymmetric inputs is not minimal if and only if the two inputs are equal. In this case the network is ODE-equivalent to a three-cell network with one (asymmetric) input.

\begin{table}
\begin{center}
   \resizebox*{!}{0.95\textheight}{{\tiny 
   
 \begin{tabular}{|c||c||c|}
 \hline 
%  &  &   \\
%  Network &  Network  & Network\\
% \hline
\begin{tikzpicture}
  [scale=.15,auto=left, node distance=1.5cm, every node/.style={circle,draw}]
 \node[fill=white] (n1) at (4,0) {\small{1}};
 \node[fill=white] (n2) at (24,0) {\small{2}}; \node[fill=white] (n3) at (14,9)  {\small{3}};
\draw[->, thick] (n1) to  [in=120,out=70,looseness=5]  (n1);
\draw[->>, thick] (n1) to  [in=205,out=155,looseness=5] (n1);
 \draw[->, thick] (n3) to  [in=110,out=160,looseness=5]  (n3);
\draw[->, thick] (n1) edge  [bend left=-10] (n2); 
 \draw[<<-, thick] (n2) edge  [bend left=10] (n3); 
\draw[->>, thick] (n2) edge  [bend left=10] (n3); 
\end{tikzpicture}
 &
 \begin{tikzpicture}
  [scale=.15,auto=left, node distance=1.5cm, every node/.style={circle,draw}]
 \node[fill=white] (n1) at (4,0) {\small{1}};
 \node[fill=white] (n2) at (24,0) {\small{2}}; \node[fill=white] (n3) at (14,9)  {\small{3}};
 \draw[->, thick] (n1) to  [in=120,out=70,looseness=5] (n1);
\draw[<<-, thick] (n1) edge  [bend left=10] (n2); 
\draw[->, thick] (n1) edge  [bend left=-10] (n2); 
 \draw[<<-, thick] (n2) edge  [bend left=10] (n3); 
 \draw[->, thick] (n2) edge  [bend left=-10] (n3); 
\draw[->>, thick] (n2) edge  [bend left=10] (n3); 
\end{tikzpicture}
&
\begin{tikzpicture}
  [scale=.15,auto=left, node distance=1.5cm, every node/.style={circle,draw}]
 \node[fill=white] (n1) at (4,0) {\small{1}};
 \node[fill=white] (n2) at (24,0) {\small{2}}; \node[fill=white] (n3) at (14,9)  {\small{3}};
 \draw[->, thick] (n1) to  [in=120,out=70,looseness=5] (n1);
\draw[<<-, thick] (n1) edge  [bend left=10] (n2); 
\draw[->>, thick] (n1) edge  [bend left=10] (n2); 
 \draw[<-, thick] (n2) edge  [bend left=-10] (n3); 
 \draw[->, thick] (n2) edge  [bend left=-10] (n3); 
\draw[->>, thick] (n2) edge  [bend left=10] (n3); 
\end{tikzpicture}
\\$E_6 \& B_1$ 
& $D_1 \& F_3$ 
&$B_1 \& F_2$ 
 \\ \hline 
 \begin{tikzpicture}
  [scale=.15,auto=left, node distance=1.5cm, every node/.style={circle,draw}]
 \node[fill=white] (n1) at (4,0) {\small{1}};
 \node[fill=white] (n2) at (24,0) {\small{2}}; \node[fill=white] (n3) at (14,9)  {\small{3}};
 \draw[->, thick] (n1) to  [in=120,out=70,looseness=5] (n1);
 \draw[->>, thick] (n3) to  [in=70,out=20,looseness=5] (n3);
\draw[<<-, thick] (n1) edge  [bend left=10] (n2); 
\draw[->, thick] (n1) edge  [bend left=-10] (n2); 
 \draw[<<-, thick] (n2) edge  [bend left=10] (n3); 
\draw[->, thick] (n2) edge  [bend left=-10] (n3); 
\end{tikzpicture}
&
 \begin{tikzpicture}
  [scale=.15,auto=left, node distance=1.5cm, every node/.style={circle,draw}]
 \node[fill=white] (n1) at (4,0) {\small{1}};
 \node[fill=white] (n2) at (24,0) {\small{2}}; \node[fill=white] (n3) at (14,9)  {\small{3}};
 \draw[->, thick] (n1) to  [in=120,out=70,looseness=5] (n1);
\draw[<<-, thick] (n1) edge  [bend left=10] (n2); 
\draw[->, thick] (n1) edge  [bend left=-10] (n2); 
 \draw[<<-, thick] (n2) edge  [bend left=10] (n3); 
\draw[<-, thick] (n3) edge [bend right=10] (n1); 
\draw[<<-, thick] (n3) edge [bend right=-10] (n1); 
\end{tikzpicture}
 &
 \begin{tikzpicture}
  [scale=.15,auto=left, node distance=1.5cm, every node/.style={circle,draw}]
 \node[fill=white] (n1) at (4,0) {\small{1}};
 \node[fill=white] (n2) at (24,0) {\small{2}}; \node[fill=white] (n3) at (14,9)  {\small{3}};
 \draw[->, thick] (n1) to  [in=120,out=70,looseness=5] (n1);
\draw[->>, thick] (n3) edge [bend right=-10] (n1); 
\draw[->, thick] (n1) edge  [bend left=-10] (n2); 
\draw[->>, thick] (n1) edge  [bend left=10] (n2); 
\draw[<<-, thick] (n3) edge [bend right=-10] (n1); 
\draw[->, thick] (n2) edge  [bend left=-10] (n3); 
\end{tikzpicture}
\\ $D_1 \& D_6$ 
& $C_1 \& A_2$ 
& $D_1 \& F_6$ 
\\ \hline 
\begin{tikzpicture}
  [scale=.15,auto=left, node distance=1.5cm, every node/.style={circle,draw}]
 \node[fill=white] (n1) at (4,0) {\small{1}};
 \node[fill=white] (n2) at (24,0) {\small{2}}; \node[fill=white] (n3) at (14,9)  {\small{3}};
\draw[<-, thick] (n1) edge  [bend left=-10] (n2); 
\draw[->>, thick] (n3) edge [bend right=-10] (n1); 
\draw[->>, thick] (n1) edge  [bend left=10] (n2); 
 \draw[<-, thick] (n2) edge  [bend left=-10] (n3); 
\draw[<-, thick] (n3) edge [bend right=10] (n1); 
\draw[->>, thick] (n2) edge  [bend left=10] (n3); 
\end{tikzpicture}
&
 \begin{tikzpicture}
  [scale=.15,auto=left, node distance=1.5cm, every node/.style={circle,draw}]
 \node[fill=white] (n1) at (4,0) {\small{1}};
 \node[fill=white] (n2) at (24,0) {\small{2}}; \node[fill=white] (n3) at (14,9)  {\small{3}};
 \draw[->, thick] (n1) to  [in=120,out=70,looseness=5] (n1);
\draw[->>, thick] (n3) edge [bend right=-10] (n1); 
\draw[->, thick] (n1) edge  [bend left=-10] (n2); 
\draw[->>, thick] (n1) edge  [bend left=10] (n2); 
 \draw[->, thick] (n2) edge  [bend left=-10] (n3); 
\draw[->>, thick] (n2) edge  [bend left=10] (n3); 
\end{tikzpicture}
&
 \begin{tikzpicture}
  [scale=.15,auto=left, node distance=1.5cm, every node/.style={circle,draw}]
 \node[fill=white] (n1) at (4,0) {\small{1}};
 \node[fill=white] (n2) at (24,0) {\small{2}}; \node[fill=white] (n3) at (14,9)  {\small{3}};
\draw[->, thick] (n1) to  [in=120,out=70,looseness=5]  (n1);
\draw[->>, thick] (n1) to  [in=205,out=155,looseness=5] (n1);
\draw[->, thick] (n1) edge  [bend left=-10] (n2); 
 \draw[<<-, thick] (n2) edge  [bend left=10] (n3); 
\draw[<<-, thick] (n3) edge [bend right=-10] (n1); 
\draw[->, thick] (n2) edge  [bend left=-10] (n3); 
\end{tikzpicture}
\\$A_2 \& A_1$ 
& $D_1 \& A_1$ 
& $D_1 \& D_2$ 
\\ \hline 
 \begin{tikzpicture}
  [scale=.15,auto=left, node distance=1.5cm, every node/.style={circle,draw}]
 \node[fill=white] (n1) at (4,0) {\small{1}};
 \node[fill=white] (n2) at (24,0) {\small{2}}; \node[fill=white] (n3) at (14,9)  {\small{3}};
 \draw[->, thick] (n1) to  [in=120,out=70,looseness=5] (n1);
 \draw[->>, thick] (n3) to  [in=70,out=20,looseness=5] (n3);
\draw[->>, thick] (n3) edge [bend right=-10] (n1); 
\draw[->, thick] (n1) edge  [bend left=-10] (n2); 
\draw[->>, thick] (n1) edge  [bend left=10] (n2); 
\draw[->, thick] (n2) edge  [bend left=-10] (n3); 
\end{tikzpicture}
&
\begin{tikzpicture}
  [scale=.15,auto=left, node distance=1.5cm, every node/.style={circle,draw}]
 \node[fill=white] (n1) at (4,0) {\small{1}};
 \node[fill=white] (n2) at (24,0) {\small{2}}; \node[fill=white] (n3) at (14,9)  {\small{3}};
\draw[->, thick] (n1) to  [in=120,out=70,looseness=5]  (n1);
\draw[->>, thick] (n1) to  [in=205,out=155,looseness=5] (n1);
\draw[->, thick] (n1) edge  [bend left=-10] (n2); 
 \draw[<<-, thick] (n2) edge  [bend left=10] (n3); 
 \draw[->, thick] (n2) edge  [bend left=-10] (n3); 
\draw[->>, thick] (n2) edge  [bend left=10] (n3); 
\end{tikzpicture}
&
 \begin{tikzpicture}
  [scale=.15,auto=left, node distance=1.5cm, every node/.style={circle,draw}]
 \node[fill=white] (n1) at (4,0) {\small{1}};
 \node[fill=white] (n2) at (24,0) {\small{2}}; \node[fill=white] (n3) at (14,9)  {\small{3}};
 \draw[->, thick] (n1) to  [in=120,out=70,looseness=5] (n1);
 \draw[->>, thick] (n2) to  [in=70,out=120,looseness=5] (n2);
\draw[->>, thick] (n3) edge [bend right=-10] (n1); 
\draw[->, thick] (n1) edge  [bend left=-10] (n2); 
\draw[<<-, thick] (n3) edge [bend right=-10] (n1); 
\draw[->, thick] (n2) edge  [bend left=-10] (n3); 
\end{tikzpicture}
\\ $D_1 \& D_5$ 
&$D_1 \& B_1$ 
& $D_1 \& B_2$ 
\\ \hline 
\begin{tikzpicture}
  [scale=.15,auto=left, node distance=1.5cm, every node/.style={circle,draw}]
 \node[fill=white] (n1) at (4,0) {\small{1}};
 \node[fill=white] (n2) at (24,0) {\small{2}}; \node[fill=white] (n3) at (14,9)  {\small{3}};
\draw[->, thick] (n1) to  [in=120,out=70,looseness=5]  (n1);
\draw[->>, thick] (n1) to  [in=205,out=155,looseness=5] (n1);
 \draw[->>, thick] (n3) to  [in=70,out=20,looseness=5] (n3);
\draw[->, thick] (n1) edge  [bend left=-10] (n2); 
 \draw[<<-, thick] (n2) edge  [bend left=10] (n3); 
\draw[->, thick] (n2) edge  [bend left=-10] (n3); 
\end{tikzpicture}
& 
 \begin{tikzpicture}
  [scale=.15,auto=left, node distance=1.5cm, every node/.style={circle,draw}]
 \node[fill=white] (n1) at (4,0) {\small{1}};
 \node[fill=white] (n2) at (24,0) {\small{2}}; \node[fill=white] (n3) at (14,9)  {\small{3}};
 \draw[->, thick] (n1) to  [in=120,out=70,looseness=5] (n1);
 \draw[->, thick] (n3) to  [in=110,out=160,looseness=5]  (n3);
\draw[<<-, thick] (n1) edge  [bend left=10] (n2); 
\draw[->, thick] (n1) edge  [bend left=-10] (n2); 
 \draw[<<-, thick] (n2) edge  [bend left=10] (n3); 
\draw[<<-, thick] (n3) edge [bend right=-10] (n1); 
\end{tikzpicture}
&
\begin{tikzpicture}
  [scale=.15,auto=left, node distance=1.5cm, every node/.style={circle,draw}]
 \node[fill=white] (n1) at (4,0) {\small{1}};
 \node[fill=white] (n2) at (24,0) {\small{2}}; \node[fill=white] (n3) at (14,9)  {\small{3}};
 \draw[->, thick] (n1) to  [in=120,out=70,looseness=5] (n1);
\draw[<<-, thick] (n1) edge  [bend left=10] (n2); 
\draw[->, thick] (n1) edge  [bend left=-10] (n2); 
 \draw[<<-, thick] (n2) edge  [bend left=10] (n3); 
\draw[<<-, thick] (n3) edge [bend right=-10] (n1); 
\draw[->, thick] (n2) edge  [bend left=-10] (n3); 
\end{tikzpicture}
\\$D_1 \& E_4$ 
& $E_6 \& A_2$ 
&$D_1 \& A_2$ 
\\ \hline 
 \begin{tikzpicture}
  [scale=.15,auto=left, node distance=1.5cm, every node/.style={circle,draw}]
 \node[fill=white] (n1) at (4,0) {\small{1}};
 \node[fill=white] (n2) at (24,0) {\small{2}}; \node[fill=white] (n3) at (14,9)  {\small{3}};
 \draw[->, thick] (n1) to  [in=120,out=70,looseness=5] (n1);
 \draw[->>, thick] (n3) to  [in=70,out=20,looseness=5] (n3);
\draw[<<-, thick] (n1) edge  [bend left=10] (n2); 
\draw[->>, thick] (n1) edge  [bend left=10] (n2); 
 \draw[<-, thick] (n2) edge  [bend left=-10] (n3); 
\draw[->, thick] (n2) edge  [bend left=-10] (n3); 
\end{tikzpicture}
&
\begin{tikzpicture}
  [scale=.15,auto=left, node distance=1.5cm, every node/.style={circle,draw}]
 \node[fill=white] (n1) at (4,0) {\small{1}};
 \node[fill=white] (n2) at (24,0) {\small{2}}; \node[fill=white] (n3) at (14,9)  {\small{3}};
 \draw[->, thick] (n1) to  [in=120,out=70,looseness=5] (n1);
\draw[<<-, thick] (n1) edge  [bend left=10] (n2); 
\draw[<-, thick] (n2) edge  [bend left=-10] (n3); 
\draw[<<-, thick] (n2) edge  [bend left=10] (n3); 
\draw[<<-, thick] (n3) edge [bend right=-10] (n1); 
\draw[->, thick] (n2) edge  [bend left=-10] (n3); 
\end{tikzpicture}
 & 
\begin{tikzpicture}
  [scale=.15,auto=left, node distance=1.5cm, every node/.style={circle,draw}]
 \node[fill=white] (n1) at (4,0) {\small{1}};
 \node[fill=white] (n2) at (24,0) {\small{2}}; \node[fill=white] (n3) at (14,9)  {\small{3}};
\draw[<-, thick] (n1) edge  [bend left=-10] (n2); 
 \draw[<<-, thick] (n1) edge  [bend left=10] (n2); 
\draw[->, thick] (n1) edge  [bend left=-10] (n2); 
 \draw[<<-, thick] (n2) edge  [bend left=10] (n3); 
\draw[<-, thick] (n3) edge [bend right=10] (n1); 
\draw[<<-, thick] (n3) edge [bend right=-10] (n1); 
\end{tikzpicture}
\\ $B_1 \& B_3$ 
&$B_1 \& A_2$ 
&$F_1 \& A_2$ 
\\ \hline 
\begin{tikzpicture}
  [scale=.15,auto=left, node distance=1.5cm, every node/.style={circle,draw}]
 \node[fill=white] (n1) at (4,0) {\small{1}};
 \node[fill=white] (n2) at (24,0) {\small{2}}; \node[fill=white] (n3) at (14,9)  {\small{3}};
\draw[<-, thick] (n1) edge  [bend left=-10] (n2); 
\draw[->>, thick] (n3) edge [bend right=-10] (n1); 
\draw[->, thick] (n1) edge  [bend left=-10] (n2); 
\draw[->>, thick] (n1) edge  [bend left=10] (n2); 
\draw[<-, thick] (n3) edge [bend right=10] (n1); 
\draw[->>, thick] (n2) edge  [bend left=10] (n3); 
\end{tikzpicture}
& & \\
$F_1 \& A_1$ 
& & \\ \hline 
\end{tabular}}}
\end{center}
 \caption{3-cell networks with two asymmetric inputs and no 2D synchrony.} 
 \label{tab:C3L0RI2.tex}
 \end{table}

\begin{table}
\begin{center}
   \resizebox*{!}{0.95\textheight}{{\tiny 
 \begin{tabular}{|c||c||c|}
 \hline 
%  &  &   \\
%Network & Network & Network  \\
% \hline
\begin{tikzpicture}
  [scale=.15,auto=left, node distance=1.5cm, every node/.style={circle,draw}]
 \node[fill=white] (n1) at (4,0) {\small{1}};
 \node[fill=white] (n2) at (24,0) {\small{2}}; \node[fill=white] (n3) at (14,9)  {\small{3}};
\draw[->, thick] (n1) to  [in=120,out=70,looseness=5]  (n1);
\draw[->>, thick] (n1) to  [in=205,out=155,looseness=5] (n1);
 \draw[->, thick] (n3) to  [in=110,out=160,looseness=5]  (n3);
 \draw[->>, thick] (n3) to  [in=70,out=20,looseness=5] (n3);
\draw[->, thick] (n1) edge  [bend left=-10] (n2); 
 \draw[<<-, thick] (n2) edge  [bend left=10] (n3); 
\end{tikzpicture}
 &
 \begin{tikzpicture}
  [scale=.15,auto=left, node distance=1.5cm, every node/.style={circle,draw}]
 \node[fill=white] (n1) at (4,0) {\small{1}};
 \node[fill=white] (n2) at (24,0) {\small{2}}; \node[fill=white] (n3) at (14,9)  {\small{3}};
\draw[->, thick] (n1) to  [in=120,out=70,looseness=5]  (n1);
\draw[->>, thick] (n1) to  [in=205,out=155,looseness=5] (n1);
 \draw[->>, thick] (n2) to  [in=70,out=120,looseness=5] (n2);
\draw[->, thick] (n1) edge  [bend left=-10] (n2); 
\draw[<<-, thick] (n3) edge [bend right=-10] (n1); 
\draw[->, thick] (n2) edge  [bend left=-10] (n3); 
\end{tikzpicture}
 &
\begin{tikzpicture}
  [scale=.15,auto=left, node distance=1.5cm, every node/.style={circle,draw}]
 \node[fill=white] (n1) at (4,0) {\small{1}};
 \node[fill=white] (n2) at (24,0) {\small{2}}; \node[fill=white] (n3) at (14,9)  {\small{3}};
\draw[->, thick] (n1) to  [in=120,out=70,looseness=5]  (n1);
\draw[->>, thick] (n1) to  [in=205,out=155,looseness=5] (n1);
\draw[->, thick] (n1) edge  [bend left=-10] (n2); 
\draw[->>, thick] (n1) edge  [bend left=10] (n2); 
\draw[<-, thick] (n3) edge [bend right=10] (n1); 
\draw[->>, thick] (n2) edge  [bend left=10] (n3); 
\end{tikzpicture}
 \\$E_6 \& E_4$  & $D_1 \& E_1$ &$C_1 \& D_1$ 
\\ \hline 
\begin{tikzpicture}
  [scale=.15,auto=left, node distance=1.5cm, every node/.style={circle,draw}]
 \node[fill=white] (n1) at (4,0) {\small{1}};
 \node[fill=white] (n2) at (24,0) {\small{2}}; \node[fill=white] (n3) at (14,9)  {\small{3}};
 \draw[->, thick] (n1) to  [in=120,out=70,looseness=5] (n1);
\draw[<<-, thick] (n1) edge  [bend left=10] (n2); 
\draw[->, thick] (n1) edge  [bend left=-10] (n2); 
\draw[->>, thick] (n1) edge  [bend left=10] (n2); 
\draw[<<-, thick] (n3) edge [bend right=-10] (n1); 
\draw[->, thick] (n2) edge  [bend left=-10] (n3); 
\end{tikzpicture}
 &
\begin{tikzpicture}
  [scale=.15,auto=left, node distance=1.5cm, every node/.style={circle,draw}]
 \node[fill=white] (n1) at (4,0) {\small{1}};
 \node[fill=white] (n2) at (24,0) {\small{2}}; \node[fill=white] (n3) at (14,9)  {\small{3}};
 \draw[->, thick] (n1) to  [in=120,out=70,looseness=5] (n1);
 \draw[->>, thick] (n2) to  [in=70,out=120,looseness=5] (n2);
\draw[<<-, thick] (n1) edge  [bend left=10] (n2); 
\draw[->, thick] (n1) edge  [bend left=-10] (n2); 
\draw[<-, thick] (n3) edge [bend right=10] (n1); 
\draw[<<-, thick] (n3) edge [bend right=-10] (n1); 
\end{tikzpicture}
 &
 \begin{tikzpicture}
  [scale=.15,auto=left, node distance=1.5cm, every node/.style={circle,draw}]
 \node[fill=white] (n1) at (4,0) {\small{1}};
 \node[fill=white] (n2) at (24,0) {\small{2}}; \node[fill=white] (n3) at (14,9)  {\small{3}};
 \draw[->, thick] (n1) to  [in=120,out=70,looseness=5] (n1);
\draw[<<-, thick] (n1) edge  [bend left=10] (n2); 
\draw[->, thick] (n1) edge  [bend left=-10] (n2); 
\draw[->>, thick] (n1) edge  [bend left=10] (n2); 
 \draw[->, thick] (n2) edge  [bend left=-10] (n3); 
\draw[->>, thick] (n2) edge  [bend left=10] (n3); 
\end{tikzpicture}
\\ $D_1 \& F_1$ 
&$C_1 \& D_4$ 
& $D_1 \& F_2$ 
\\\hline 
\begin{tikzpicture}
  [scale=.15,auto=left, node distance=1.5cm, every node/.style={circle,draw}]
 \node[fill=white] (n1) at (4,0) {\small{1}};
 \node[fill=white] (n2) at (24,0) {\small{2}}; \node[fill=white] (n3) at (14,9)  {\small{3}};
 \draw[->, thick] (n1) to  [in=120,out=70,looseness=5] (n1);
 \draw[->>, thick] (n3) to  [in=70,out=20,looseness=5] (n3);
\draw[<<-, thick] (n1) edge  [bend left=10] (n2); 
\draw[->, thick] (n1) edge  [bend left=-10] (n2); 
 \draw[<<-, thick] (n2) edge  [bend left=10] (n3); 
\draw[<-, thick] (n3) edge [bend right=10] (n1); 
\end{tikzpicture}
 &
 \begin{tikzpicture}
  [scale=.15,auto=left, node distance=1.5cm, every node/.style={circle,draw}]
 \node[fill=white] (n1) at (4,0) {\small{1}};
 \node[fill=white] (n2) at (24,0) {\small{2}}; \node[fill=white] (n3) at (14,9)  {\small{3}};
 \draw[->, thick] (n1) to  [in=120,out=70,looseness=5] (n1);
 \draw[->>, thick] (n3) to  [in=70,out=20,looseness=5] (n3);
\draw[<<-, thick] (n1) edge  [bend left=10] (n2); 
\draw[->, thick] (n1) edge  [bend left=-10] (n2); 
\draw[->>, thick] (n1) edge  [bend left=10] (n2); 
\draw[->, thick] (n2) edge  [bend left=-10] (n3); 
\end{tikzpicture}
 &
\begin{tikzpicture}
  [scale=.15,auto=left, node distance=1.5cm, every node/.style={circle,draw}]
 \node[fill=white] (n1) at (4,0) {\small{1}};
 \node[fill=white] (n2) at (24,0) {\small{2}}; \node[fill=white] (n3) at (14,9)  {\small{3}};
 \draw[->, thick] (n1) to  [in=120,out=70,looseness=5] (n1);
 \draw[->>, thick] (n2) to  [in=70,out=120,looseness=5] (n2);
\draw[<<-, thick] (n1) edge  [bend left=10] (n2); 
\draw[->, thick] (n1) edge  [bend left=-10] (n2); 
\draw[<<-, thick] (n3) edge [bend right=-10] (n1); 
\draw[->, thick] (n2) edge  [bend left=-10] (n3); 
\end{tikzpicture}
\\$C_1 \& D_6$ 
& $D_1 \& B_3$ 
&$D_1 \& D_4$ 
 \\ \hline 
 \begin{tikzpicture}
  [scale=.15,auto=left, node distance=1.5cm, every node/.style={circle,draw}]
 \node[fill=white] (n1) at (4,0) {\small{1}};
 \node[fill=white] (n2) at (24,0) {\small{2}}; \node[fill=white] (n3) at (14,9)  {\small{3}};
 \draw[->, thick] (n1) to  [in=120,out=70,looseness=5] (n1);
\draw[->>, thick] (n3) edge [bend right=-10] (n1); 
\draw[->, thick] (n1) edge  [bend left=-10] (n2); 
 \draw[<<-, thick] (n2) edge  [bend left=10] (n3); 
\draw[<<-, thick] (n3) edge [bend right=-10] (n1); 
\draw[->, thick] (n2) edge  [bend left=-10] (n3); 
\end{tikzpicture}
 &  
\begin{tikzpicture}
  [scale=.15,auto=left, node distance=1.5cm, every node/.style={circle,draw}]
 \node[fill=white] (n1) at (4,0) {\small{1}};
 \node[fill=white] (n2) at (24,0) {\small{2}}; \node[fill=white] (n3) at (14,9)  {\small{3}};
\draw[->, thick] (n1) to  [in=120,out=70,looseness=5]  (n1);
\draw[->>, thick] (n1) to  [in=205,out=155,looseness=5] (n1);
\draw[->, thick] (n1) edge  [bend left=-10] (n2); 
 \draw[<<-, thick] (n2) edge  [bend left=10] (n3); 
\draw[<-, thick] (n3) edge [bend right=10] (n1); 
\draw[->>, thick] (n2) edge  [bend left=10] (n3); 
\end{tikzpicture}
 &  
 \begin{tikzpicture}
  [scale=.15,auto=left, node distance=1.5cm, every node/.style={circle,draw}]
 \node[fill=white] (n1) at (4,0) {\small{1}};
 \node[fill=white] (n2) at (24,0) {\small{2}}; \node[fill=white] (n3) at (14,9)  {\small{3}};
 \draw[->, thick] (n1) to  [in=120,out=70,looseness=5] (n1);
\draw[->>, thick] (n3) edge [bend right=-10] (n1); 
\draw[->, thick] (n1) edge  [bend left=-10] (n2); 
 \draw[<<-, thick] (n2) edge  [bend left=10] (n3); 
 \draw[->, thick] (n2) edge  [bend left=-10] (n3); 
\draw[->>, thick] (n2) edge  [bend left=10] (n3); 
\end{tikzpicture}
\\ $D_1 \& F_5$ 
&$C_1 \& B_1$ 
& $D_1 \& F_4$ 
\\ \hline 
\begin{tikzpicture}
  [scale=.15,auto=left, node distance=1.5cm, every node/.style={circle,draw}]
 \node[fill=white] (n1) at (4,0) {\small{1}};
 \node[fill=white] (n2) at (24,0) {\small{2}}; \node[fill=white] (n3) at (14,9)  {\small{3}};
 \draw[->, thick] (n1) to  [in=120,out=70,looseness=5] (n1);
 \draw[->>, thick] (n3) to  [in=70,out=20,looseness=5] (n3);
\draw[<<-, thick] (n1) edge  [bend left=10] (n2); 
\draw[->, thick] (n1) edge  [bend left=-10] (n2); 
\draw[->>, thick] (n1) edge  [bend left=10] (n2); 
\draw[<-, thick] (n3) edge [bend right=10] (n1); 
\end{tikzpicture}
 &
 \begin{tikzpicture}
  [scale=.15,auto=left, node distance=1.5cm, every node/.style={circle,draw}]
 \node[fill=white] (n1) at (4,0) {\small{1}};
 \node[fill=white] (n2) at (24,0) {\small{2}}; \node[fill=white] (n3) at (14,9)  {\small{3}};
 \draw[->, thick] (n1) to  [in=120,out=70,looseness=5] (n1);
 \draw[->, thick] (n3) to  [in=110,out=160,looseness=5]  (n3);
\draw[<<-, thick] (n1) edge  [bend left=10] (n2); 
\draw[->, thick] (n1) edge  [bend left=-10] (n2); 
 \draw[<<-, thick] (n2) edge  [bend left=10] (n3); 
\draw[->>, thick] (n2) edge  [bend left=10] (n3); 
\end{tikzpicture}
 &  
\begin{tikzpicture}
  [scale=.15,auto=left, node distance=1.5cm, every node/.style={circle,draw}]
 \node[fill=white] (n1) at (4,0) {\small{1}};
 \node[fill=white] (n2) at (24,0) {\small{2}}; \node[fill=white] (n3) at (14,9)  {\small{3}};
\draw[->, thick] (n1) to  [in=120,out=70,looseness=5]  (n1);
\draw[->>, thick] (n1) to  [in=205,out=155,looseness=5] (n1);
 \draw[->>, thick] (n3) to  [in=70,out=20,looseness=5] (n3);
\draw[->, thick] (n1) edge  [bend left=-10] (n2); 
\draw[->>, thick] (n1) edge  [bend left=10] (n2); 
\draw[->, thick] (n2) edge  [bend left=-10] (n3); 
\end{tikzpicture}
\\$C_1 \& B_3$ 
& $E_6 \& F_3$ 
&$D_1 \& E_6$ 
 \\ \hline 
 \begin{tikzpicture}
  [scale=.15,auto=left, node distance=1.5cm, every node/.style={circle,draw}]
 \node[fill=white] (n1) at (4,0) {\small{1}};
 \node[fill=white] (n2) at (24,0) {\small{2}}; \node[fill=white] (n3) at (14,9)  {\small{3}};
 \draw[->, thick] (n1) to  [in=120,out=70,looseness=5] (n1);
 \draw[->, thick] (n3) to  [in=110,out=160,looseness=5]  (n3);
\draw[->>, thick] (n3) edge [bend right=-10] (n1); 
\draw[->, thick] (n1) edge  [bend left=-10] (n2); 
\draw[->>, thick] (n1) edge  [bend left=10] (n2); 
\draw[<<-, thick] (n3) edge [bend right=-10] (n1); 
\end{tikzpicture}
 &  
 \begin{tikzpicture}
  [scale=.15,auto=left, node distance=1.5cm, every node/.style={circle,draw}]
 \node[fill=white] (n1) at (4,0) {\small{1}};
 \node[fill=white] (n2) at (24,0) {\small{2}}; \node[fill=white] (n3) at (14,9)  {\small{3}};
 \draw[->, thick] (n1) to  [in=120,out=70,looseness=5] (n1);
 \draw[->, thick] (n3) to  [in=110,out=160,looseness=5]  (n3);
\draw[->>, thick] (n3) edge [bend right=-10] (n1); 
\draw[->, thick] (n1) edge  [bend left=-10] (n2); 
 \draw[<<-, thick] (n2) edge  [bend left=10] (n3); 
\draw[->>, thick] (n2) edge  [bend left=10] (n3); 
\end{tikzpicture}
 &  
\begin{tikzpicture}
  [scale=.15,auto=left, node distance=1.5cm, every node/.style={circle,draw}]
 \node[fill=white] (n1) at (4,0) {\small{1}};
 \node[fill=white] (n2) at (24,0) {\small{2}}; \node[fill=white] (n3) at (14,9)  {\small{3}};
 \draw[->, thick] (n1) to  [in=120,out=70,looseness=5] (n1);
\draw[<<-, thick] (n1) edge  [bend left=10] (n2); 
\draw[->>, thick] (n1) edge  [bend left=10] (n2); 
 \draw[<-, thick] (n2) edge  [bend left=-10] (n3); 
\draw[<<-, thick] (n3) edge [bend right=-10] (n1); 
\draw[->, thick] (n2) edge  [bend left=-10] (n3); 
\end{tikzpicture}
\\ $E_6 \& F_6$ 
& $E_6 \& F_4$ 
&$B_1 \& F_1$ 
\\ \hline 
\begin{tikzpicture}
  [scale=.15,auto=left, node distance=1.5cm, every node/.style={circle,draw}]
 \node[fill=white] (n1) at (4,0) {\small{1}};
 \node[fill=white] (n2) at (24,0) {\small{2}}; \node[fill=white] (n3) at (14,9)  {\small{3}};
\draw[<-, thick] (n1) edge  [bend left=-10] (n2); 
 \draw[<<-, thick] (n1) edge  [bend left=10] (n2); 
\draw[->, thick] (n1) edge  [bend left=-10] (n2); 
\draw[->>, thick] (n1) edge  [bend left=10] (n2); 
\draw[<-, thick] (n3) edge [bend right=10] (n1); 
\draw[->>, thick] (n2) edge  [bend left=10] (n3); 
\end{tikzpicture}
& 
\begin{tikzpicture}
  [scale=.15,auto=left, node distance=1.5cm, every node/.style={circle,draw}]
 \node[fill=white] (n1) at (4,0) {\small{1}};
 \node[fill=white] (n2) at (24,0) {\small{2}}; \node[fill=white] (n3) at (14,9)  {\small{3}};
\draw[<-, thick] (n1) edge  [bend left=-10] (n2); 
 \draw[<<-, thick] (n1) edge  [bend left=10] (n2); 
\draw[->, thick] (n1) edge  [bend left=-10] (n2); 
 \draw[<<-, thick] (n2) edge  [bend left=10] (n3); 
\draw[<-, thick] (n3) edge [bend right=10] (n1); 
\draw[->>, thick] (n2) edge  [bend left=10] (n3); 
\end{tikzpicture}
 &  
\begin{tikzpicture}
  [scale=.15,auto=left, node distance=1.5cm, every node/.style={circle,draw}]
 \node[fill=white] (n1) at (4,0) {\small{1}};
 \node[fill=white] (n2) at (24,0) {\small{2}}; \node[fill=white] (n3) at (14,9)  {\small{3}};
\draw[<-, thick] (n1) edge  [bend left=-10] (n2); 
\draw[->>, thick] (n3) edge [bend right=-10] (n1); 
\draw[->, thick] (n1) edge  [bend left=-10] (n2); 
\draw[->>, thick] (n1) edge  [bend left=10] (n2); 
\draw[<-, thick] (n3) edge [bend right=10] (n1); 
\draw[<<-, thick] (n3) edge [bend right=-10] (n1); 
\end{tikzpicture}
\\$F_1 \& F_2$ 
&$F_1 \& F_3$ 
&$F_1 \& F_6$ 
\\ \hline 
\end{tabular}
 }}
 \end{center}
 \caption{3-cell networks with two asymmetric inputs and one 2D synchrony.} 
 \label{tab:C3L1.tex}
 \end{table}

\begin{table}
 \begin{center}
 {\tiny 
 \begin{tabular}{|c||c||c|}
 \hline 
% & &    \\
% Network & Network &  Network \\
% \hline
\begin{tikzpicture}
  [scale=.15,auto=left, node distance=1.5cm, every node/.style={circle,draw}]
 \node[fill=white] (n1) at (4,0) {\small{1}};
 \node[fill=white] (n2) at (24,0) {\small{2}}; \node[fill=white] (n3) at (14,9)  {\small{3}};
\draw[->, thick] (n1) to  [in=120,out=70,looseness=5]  (n1);
\draw[->>, thick] (n1) to  [in=205,out=155,looseness=5] (n1);
 \draw[->>, thick] (n3) to  [in=70,out=20,looseness=5] (n3);
\draw[->, thick] (n1) edge  [bend left=-10] (n2); 
\draw[->>, thick] (n1) edge  [bend left=10] (n2); 
\draw[<-, thick] (n3) edge [bend right=10] (n1); 
\end{tikzpicture}
 &  
 
\begin{tikzpicture}
  [scale=.15,auto=left, node distance=1.5cm, every node/.style={circle,draw}]
 \node[fill=white] (n1) at (4,0) {\small{1}};
 \node[fill=white] (n2) at (24,0) {\small{2}}; \node[fill=white] (n3) at (14,9)  {\small{3}};
\draw[->, thick] (n1) to  [in=120,out=70,looseness=5]  (n1);
\draw[->>, thick] (n1) to  [in=205,out=155,looseness=5] (n1);
 \draw[->>, thick] (n2) to  [in=70,out=120,looseness=5] (n2);
\draw[->, thick] (n1) edge  [bend left=-10] (n2); 
\draw[<-, thick] (n3) edge [bend right=10] (n1); 
\draw[->>, thick] (n2) edge  [bend left=10] (n3); 
\end{tikzpicture}
 &

\begin{tikzpicture}
  [scale=.15,auto=left, node distance=1.5cm, every node/.style={circle,draw}]
 \node[fill=white] (n1) at (4,0) {\small{1}};
 \node[fill=white] (n2) at (24,0) {\small{2}}; \node[fill=white] (n3) at (14,9)  {\small{3}};
 \draw[->, thick] (n1) to  [in=120,out=70,looseness=5] (n1);
\draw[<<-, thick] (n1) edge  [bend left=10] (n2); 
\draw[->, thick] (n1) edge  [bend left=-10] (n2); 
\draw[->>, thick] (n1) edge  [bend left=10] (n2); 
\draw[<-, thick] (n3) edge [bend right=10] (n1); 
\draw[<<-, thick] (n3) edge [bend right=-10] (n1); 
\end{tikzpicture}
 \\ $C_1 \& E_6$ & $C_1 \& E_3$&$C_1 \& F_1$ \\
\hline 
 
\begin{tikzpicture}
  [scale=.15,auto=left, node distance=1.5cm, every node/.style={circle,draw}]
 \node[fill=white] (n1) at (4,0) {\small{1}};
 \node[fill=white] (n2) at (24,0) {\small{2}}; \node[fill=white] (n3) at (14,9)  {\small{3}};
 \draw[->, thick] (n1) to  [in=120,out=70,looseness=5] (n1);
\draw[<<-, thick] (n1) edge  [bend left=10] (n2); 
\draw[->, thick] (n1) edge  [bend left=-10] (n2); 
\draw[->>, thick] (n1) edge  [bend left=10] (n2); 
\draw[<-, thick] (n3) edge [bend right=10] (n1); 
\draw[->>, thick] (n2) edge  [bend left=10] (n3); 
\end{tikzpicture}
 & 
\begin{tikzpicture}
  [scale=.15,auto=left, node distance=1.5cm, every node/.style={circle,draw}]
 \node[fill=white] (n1) at (4,0) {\small{1}};
 \node[fill=white] (n2) at (24,0) {\small{2}}; \node[fill=white] (n3) at (14,9)  {\small{3}};
 \draw[->, thick] (n1) to  [in=120,out=70,looseness=5] (n1);
\draw[<<-, thick] (n1) edge  [bend left=10] (n2); 
\draw[->, thick] (n1) edge  [bend left=-10] (n2); 
 \draw[<<-, thick] (n2) edge  [bend left=10] (n3); 
\draw[<-, thick] (n3) edge [bend right=10] (n1); 
\draw[->>, thick] (n2) edge  [bend left=10] (n3); 
\end{tikzpicture}
&
\begin{tikzpicture}
  [scale=.15,auto=left, node distance=1.5cm, every node/.style={circle,draw}]
 \node[fill=white] (n1) at (4,0) {\small{1}};
 \node[fill=white] (n2) at (24,0) {\small{2}}; \node[fill=white] (n3) at (14,9)  {\small{3}};
 \draw[->, thick] (n1) to  [in=120,out=70,looseness=5] (n1);
 \draw[->, thick] (n3) to  [in=110,out=160,looseness=5]  (n3);
\draw[->>, thick] (n3) edge [bend right=-10] (n1); 
\draw[->, thick] (n1) edge  [bend left=-10] (n2); 
 \draw[<<-, thick] (n2) edge  [bend left=10] (n3); 
\draw[<<-, thick] (n3) edge [bend right=-10] (n1); 
\end{tikzpicture}
\\ $C_1 \& F_2$&$C_1 \& F_3$ &$E_6 \& F_5$ 
\\\hline 

\begin{tikzpicture}
  [scale=.15,auto=left, node distance=1.5cm, every node/.style={circle,draw}]
 \node[fill=white] (n1) at (4,0) {\small{1}};
 \node[fill=white] (n2) at (24,0) {\small{2}}; \node[fill=white] (n3) at (14,9)  {\small{3}};
\draw[<-, thick] (n1) edge  [bend left=-10] (n2); 
\draw[->>, thick] (n3) edge [bend right=-10] (n1); 
\draw[->, thick] (n1) edge  [bend left=-10] (n2); 
 \draw[<<-, thick] (n2) edge  [bend left=10] (n3); 
\draw[<-, thick] (n3) edge [bend right=10] (n1); 
\draw[->>, thick] (n2) edge  [bend left=10] (n3); 
\end{tikzpicture}
 & &  \\ 
$F_1 \& F_4$ & & \\\hline 
\end{tabular}
}
 \end{center}
 \caption{$3$-cell networks with two asymmetric inputs and two 2D synchrony.} 
 \label{tab:C3L2.tex}
 \end{table}

\begin{table}
 \begin{center}
 {\tiny 
 \begin{tabular}{|c|}
 \hline 
 %  \\
 % Network  \\
 %\hline

\begin{tikzpicture}
  [scale=.15,auto=left, node distance=1.5cm, every node/.style={circle,draw}]
 \node[fill=white] (n1) at (4,0) {\small{1}};
 \node[fill=white] (n2) at (24,0) {\small{2}}; \node[fill=white] (n3) at (14,9)  {\small{3}};
 \draw[->, thick] (n1) to  [in=120,out=70,looseness=5] (n1);
 \draw[->>, thick] (n2) to  [in=70,out=120,looseness=5] (n2);
\draw[<<-, thick] (n1) edge  [bend left=10] (n2); 
\draw[->, thick] (n1) edge  [bend left=-10] (n2); 
\draw[<-, thick] (n3) edge [bend right=10] (n1); 
\draw[->>, thick] (n2) edge  [bend left=10] (n3); 
\end{tikzpicture}
\\ $C_1 \& C_2$ \\\hline 
\end{tabular}}
 \end{center}
 \caption{$3$-cell network with two asymmetric inputs and three 2D synchrony.} 
 \label{tab:C2L3.tex}
 \end{table}

\begin{thm}\label{thm:enumeration}
Up to ODE-equivalence, there are $48$ minimal $3$-cell connected networks with two asymmetric inputs. See Tables~\ref{tab:C3L0RI2.tex}-\ref{tab:C2L3.tex}. 
\end{thm}

\begin{proof} 
Excluding the network where each cell receives only one self-loop, there are $26$  networks with three cells and one (asymmetric) input, which are listed in Table~\ref{tab:listaval13cell}. 
It follows then, from Corollary~\ref{cor:notmin}, that there are 
$26 \times 25 = 650$ minimal networks with three cells and two asymmetric inputs. Since we are interested in minimal networks, up to ODE-equivalence, we consider the networks up to interchange of the edge types, which gives 325 networks. 
Among the networks with one (asymmetric) input in Table~\ref{tab:listaval13cell}, 
there are two (networks $A_1$ and $A_2$) with $\ZZ_3$-symmetry, six with (networks $B_i$ and $C_i$, with $i=1,2,3$) with $\ZZ_2$-symmetry and the remaining 18 networks have no symmetry. Thus, among the 325 networks with two asymmetric inputs, up to re-enumeration of the cells, there are 64 networks, as we explain next.
When considering the union of networks $A$ with networks $A, B, C, D, E, F$, since we are interested in networks up to re-enumeration of the cells, 
we can consider only the union of network $A_1$ with networks $A_2, B, C, D, E, F$. Given the $\ZZ_3$-symmetry of $A_1$, the $\ZZ_2$-symmetry of networks $B$ and $C$ and no symmetry of networks $D, E, F$, up to re-enumeration of the cells, we get, respectively, $1,1,1,2,2,2$ networks.
When considering the union of networks $B$ with networks $B, C, D, E, F$, since we are interested in networks up to re-enumeration of the cells, 
we can consider only the union of network $B_1$ with networks $B_2, B_3, C, D, E, F$. Given the $\ZZ_2$-symmetry of networks $B$ and $C$ and no symmetry of networks $D, E, F$, up to re-enumeration of the cells, we get, respectively, $1,2,3,3,3$ networks.
When considering the union of networks $C$ with networks $C, D, E, F$, since we are interested in networks up to re-enumeration of the cells, 
we can consider only the union of network $C_1$ with networks $C_2, C_3, D, E, F$. Given the $\ZZ_2$-symmetry of networks $C$ and no symmetry of networks $D, E, F$, up to re-enumeration of the cells, we get, respectively, $1,3,3,3$ networks.
When considering the union of networks $D$ with networks $D, E, F$, since we are interested in networks up to re-enumeration of the cells, 
we consider only the union of network $D_1$ with networks  $D_2, D_3, D_4, D_5, D_6, E, F$.
Since the networks $D, E, F$ have no symmetry we get, respectively, $5,6,6$ networks.
Analogously, making the union of networks $E$ with networks $E, F$ we get, respectively, $5,6$ networks and  making the union networks $F$ with networks $F$ we get $5$ networks.
From the set of these 64 networks, we consider the bigger subset of the connected networks that are not ODE-equivalent: 
using MATLAB we obtain the 48 minimal three-cell networks with two asymmetric inputs listed in Tables~\ref{tab:C3L0RI2.tex}-\ref{tab:C2L3.tex}.  
\end{proof}

\begin{thm}\label{thm:classification}
Among the $48$ minimal three-cell connected networks with two asymmetric inputs given by Theorem~\ref{thm:enumeration}, there are $19$ networks with no two-dimensional synchrony subspaces (see Table~\ref{tab:C3L0RI2.tex}), $21$ networks with one two-dimensional synchrony subspace (see Table~\ref{tab:C3L1.tex}), $7$ networks with two two-dimensional synchrony subspaces (see Table~\ref{tab:C3L2.tex}) and one network with three two-dimensional synchrony subspaces (see Table~\ref{tab:C2L3.tex}).
\end{thm}

\begin{proof} Let $G$ be a minimal three-cell connected network with two asymmetric inputs. Then, $G=G_1 \cup G_2$ with $G_1$ and $G_2$ three-cell networks with one (asymmetric) input, both in Table~\ref{tab:listaval13cell}. 
The network $G$ has a synchrony subspace $\Delta_i$ if and only if $\Delta_i$ 
 is a synchrony subspace for both networks $G_1$ and $G_2$. 
Using the information in Table~\ref{tab:listaval13cell}, we obtain the information above stated concerning the synchrony spaces of the minimal three-cell connected networks with two asymmetric inputs.
\end{proof}

\begin{rem} 
If, among the $48$ minimal three-cell connected networks with two asymmetric inputs given by Theorem~\ref{thm:enumeration}, we consider only the 
strongly connected ones, that have one or two two-dimensional network synchrony subspace, then we see that  there are only $8$ 
networks with one two-dimensional synchrony subspace ($C_1\& D_6$, $D_1\& F_5$, $D_1\& F_4$, $E_6\& F_3$, $E_6\& F_4$, $B_1\& F_1$, $F_1\& F_3$, $ F_1\& F_6$ from Table~\ref{tab:C3L1.tex}) and $2$ networks with two two-dimensional synchrony subspaces ($C_1\&F_3$, $F_1\&F_4$ from Table~\ref{tab:C3L2.tex}). These are in accordance with the results of Aguiar, Ashwin, Dias and Field~\cite{AADF11}  concerning 
strongly connected networks of three cells and two asymmetric inputs that have one or two  two-dimensional synchrony subspace.
\hfill
$\Box$
\end{rem}

\subsection{ODE distinct three-cell two-input asymmetric networks with the same hidden symmetries}

 Rink and Sanders~\cite{NRS16,RS15} show that networks with asymmetric inputs have hidden symmetries which influence the network dynamics and moreover,  can be used to study the dynamics.  When the network has a semigroup structure, Rink and Sanders in \cite{RS15}  have calculated normal forms of coupled cell systems and in \cite{RS14} have used the hidden symmetries of the network to derive Lyapunov-Schmidt reduction that preserves hidden symmetries.  In ~\cite{NRS16}, Nijholt, Rink and Sanders have introduced the concept of fundamental network which reveals the hidden symmetries of a network. A fundamental network is a Cayley Graph of a monoid (semigroup with unity). The dynamics associated to a fundamental network can be studied using the revealed hidden symmetries and be related with the dynamics associated to the original  network which does not need to be fundamental~\cite[Theorem 3.7 \& Remark 3.9]{RS14}.

In Section 7 of \cite{RS14}, it is considered fundamental networks with two or three cells and their possible  generic codimension-one steady-state bifurcations that can occur assuming that the cell phase spaces are one-dimensional. It is remarked that these systems are fully characterized by their monoid symmetry, moreover, their semigroup representations split as the sum of mutually nonisomorphic indecomposable representations. In their classification, in case of monoid networks with three cells, it is used the fact that, there are up to isomorphism, precisely seven monoids with three elements (see \cite{C09}).

  In this section, we make two observations. We first remark that from the 48 networks with three cells and two asymmetric inputs obtained in Theorem~\ref{thm:enumeration}, there are only eight networks which have symmetry monoids with three elements. Moreover, only seven of these are fundamental networks, where all the possible seven monoids with three elements occur in this list of eight networks. The other 40 networks have symmetry monoids with more than three elements. The second remark concerns the fact that there are ODE distinct three-cell networks with the same symmetry monoid of three elements. The multiplication operation is given by the composition of such functions.

In what follows, a three-cell network with two asymmetric inputs denoted by $G_1 \& G_2$, has each  edge type $j$, for $j=1,2$, represented by a function $\sigma_j : \{1,2,3\} \to \{1,2,3\}$ such that $\sigma_j(l) = a_l$, for $l = 1,2,3$, and we represent it by 
$\sigma_j = \left[a_1\, a_2\,  a_3\right]$. Thus, if we take the edge type $j$ and $\sigma_j(l) = a_l$, then there is an edge of the type $j$ from cell $a_l$ to cell $l$ which corresponds to an edge from cell $a_l$ to cell $l$ in the network $G_j$. 

\begin{prop} \label{prop:eightfund}
From the 48 networks with three cells and two asymmetric inputs obtained in Theorem~\ref{thm:enumeration}, only eight have symmetry monoids with three elements: $A_2 \& A_1$, $E_6 \& E_4$, $C_1 \& D_1$, $C_1 \& B_1$, $E_6 \& F_5$, $C_1 \& C_2$, $C_1 \& E_3$ and $C_1 \& E_6$. Each corresponds to one of  the seven distinct  possible symmetry monoids with three elements, except the last two that have the same symmetry monoid. Except the network $C_1 \& E_6$, the other seven are fundamental networks. See Tables~\ref{tab:fund3monoids}-\ref{tabelamonoides}. 
\end{prop}

\begin{proof}
The symmetry monoid of each $G_1 \& G_2$ in the list of the 48 networks with 3 cells and two asymmetric inputs in Tables~\ref{tab:C3L0RI2.tex}-\ref{tab:C2L3.tex} is determined by three functions: $\sigma_0 = \mbox{Id}_3$ and $\sigma_1, \sigma_2$ corresponding to  the subnetworks with one input, $G_1$ and $G_2$, respectively. Except for the eight networks ($A_2 \& A_1$, $E_6 \& E_4$, $C_1 \& D_1$, $C_1 \& B_1$, $E_6 \& F_5$, $C_1 \& C_2$, $C_1 \& E_3$ and $C_1 \& E_6$), the set $\Sigma = \{ \sigma_0, \sigma_1, \sigma_2\}$ is not closed for the composition. In fact, for those 40 networks, at least one of the products $\sigma_1 \sigma_2$ or $\sigma_2 \sigma_1$ does not belong to $\Sigma$. Now, for the other eight networks, we see that $\Sigma = \{ \sigma_0, \sigma_1, \sigma_2\}$ is closed under multiplication (composition) and we have all the possibilities for the products $\sigma_i \sigma_j$ where $i, j\not=1,2$. See Tables~\ref{tab:fund3monoids}-\ref{tabelamonoides} for the matching between each of the eight networks and the corresponding symmetry monoid.  As an example, if we take the network $A_2 \& A_1$, we have that 
$$
\sigma_0 = \left[ 1\, 2\, 3\right],\quad \sigma_1 = \left[ 2\, 3\, 1\right],\quad \sigma_2 = \left[ 3\, 1\, 2\right]\, .
$$
It follows that $\Sigma = \{ \sigma_0, \sigma_1, \sigma_2\}$  is a monoid. Moreover, as $\sigma_1^2 = \sigma_2$, $\sigma_2^2 = \sigma_1$ and $\sigma_1 \sigma_2 = \sigma_2 \sigma_1 = \sigma_0$, we have that the multiplication table for $\Sigma$ corresponds to $\Sigma_6$ in Table \ref{tabelamonoides} (it corresponds to the $\Sigma_6$ in Section 7 of \cite{RS14}).  
\end{proof}

\begin{table} \label{tab:fund3monoids}
\begin{tabular}{c|c|c} 
\hline 
Network          & Monoid                                                                    & Monoid \\
                       & symmetries                                                              & structure\\
  \hline 
$A_2 \& A_1$ & $\sigma_0 = \left[ 1\, 2\, 3\right],\, \sigma_1 = \left[ 2\, 3\, 1\right],\, \sigma_2 = \left[ 3\, 1\, 2\right]$ & $\Sigma_6$\\
$E_6 \& E_4$ & $\sigma_0 = \left[ 1\, 2\, 3\right],\, \sigma_1 = \left[  1\, 1\, 3\right],\, \sigma_2 = \left[ 1\, 3\, 3\right]$ & $\Sigma_5$\\
$C_1 \& D_1$ & $\sigma_0 = \left[ 1\, 2\, 3\right],\, \sigma_1 = \left[  1\, 1\, 1\right],\, \sigma_2 = \left[ 1\, 1\, 2\right]$ & $\Sigma_1$\\
$C_1 \& B_1$ & $\sigma_0 = \left[ 1\, 2\, 3\right],\, \sigma_1 = \left[ 1\, 1\, 1\right],\, \sigma_2 = \left[ 1\, 3\, 2\right]$ & $\Sigma_7$\\
$E_6 \& F_5$ & $\sigma_0 = \left[ 1\, 2\, 3\right],\, \sigma_1 = \left[ 1\, 1\, 3\right],\, \sigma_2 = \left[ 3\, 3\, 1\right]$ & $\Sigma_2$\\
$C_1 \& C_2$ & $\sigma_0 = \left[ 1\, 2\, 3\right],\, \sigma_1 = \left[ \, 1\, 1\, 1\right],\, \sigma_2 = \left[ 2\, 2\, 2\right]$ & $\Sigma_4$\\
$C_1 \& E_3$ & $\sigma_0 = \left[ 1\, 2\, 3\right],\, \sigma_1 = \left[ 1\, 1\, 1\right],\, \sigma_2 = \left[ 1\, 2\, 2\right]$ & $\Sigma_3$\\
$C_1 \& E_6$ & $\sigma_0 = \left[ 1\, 2\, 3\right],\, \sigma_1 = \left[  1\, 1\, 1\right],\, \sigma_2 = \left[ 1\, 1\, 3\right]$ & $\Sigma_3$\\
\hline 
\end{tabular}
\caption{The eight ODE-distinct networks with three cells and two asymmetric inputs which have symmetry monoids with three elements, and the corresponding symmetry monoids. 
The monoids $\Sigma_i$ for $i=1, \ldots, 7$ appear in Table~\ref{tabelamonoides}.
Except  $C_1 \& E_6$, they are fundamental networks. Here, $\sigma_0$ represents the dependence of each cell on its own state which we omit in the network representation.}
\end{table}

\begin{table} 
\begin{tabular}{llll}
\begin{tabular}{l|lll}
$\Sigma_1$ & $\sigma_0$ & $\sigma_1$ & $\sigma_2$ \\
\hline 
$\sigma_0$ &   $\sigma_0$ & $\sigma_1$ & $\sigma_2$\\
$\sigma_1$ &  $\sigma_1$ & $\sigma_1$ & $\sigma_1$\\
$\sigma_2$ &  $\sigma_2$ & $\sigma_1$ & $\sigma_1$\\
\hline 
\end{tabular} 
&
\begin{tabular}{l|lll}
$\Sigma_2$ & $\sigma_0$ & $\sigma_1$ & $\sigma_2$ \\
\hline 
$\sigma_0$ &   $\sigma_0$ & $\sigma_1$ & $\sigma_2$\\
$\sigma_1$ &  $\sigma_1$ & $\sigma_1$ & $\sigma_2$\\
$\sigma_2$ &  $\sigma_2$ & $\sigma_2$ & $\sigma_1$\\
\hline 
\end{tabular} 
&
\begin{tabular}{l|lll}
$\Sigma_3$ & $\sigma_0$ & $\sigma_1$ & $\sigma_2$ \\
\hline 
$\sigma_0$ &   $\sigma_0$ & $\sigma_1$ & $\sigma_2$\\
$\sigma_1$ &  $\sigma_1$ & $\sigma_1$ & $\sigma_1$\\
$\sigma_2$ &  $\sigma_2$ & $\sigma_1$ & $\sigma_2$\\
\hline 
\end{tabular}  
&
 \begin{tabular}{l|lll}
$\Sigma_4$ & $\sigma_0$ & $\sigma_1$ & $\sigma_2$ \\
\hline 
$\sigma_0$ &   $\sigma_0$ & $\sigma_1$ & $\sigma_2$\\
$\sigma_1$ &  $\sigma_1$ & $\sigma_1$ & $\sigma_1$\\
$\sigma_2$ &  $\sigma_2$ & $\sigma_2$ & $\sigma_2$\\
\hline 
\end{tabular}  \\
 & & & \\
\begin{tabular}{l|lll}
$\Sigma_5$ & $\sigma_0$ & $\sigma_1$ & $\sigma_2$ \\
\hline 
$\sigma_0$ &   $\sigma_0$ & $\sigma_1$ & $\sigma_2$\\
$\sigma_1$ &  $\sigma_1$ & $\sigma_1$ & $\sigma_2$\\
$\sigma_2$ &  $\sigma_2$ & $\sigma_1$ & $\sigma_2$\\
\hline 
\end{tabular} 
&
\begin{tabular}{l|lll}
$\Sigma_6$ & $\sigma_0$ & $\sigma_1$ & $\sigma_2$ \\
\hline 
$\sigma_0$ &   $\sigma_0$ & $\sigma_1$ & $\sigma_2$\\
$\sigma_1$ &  $\sigma_1$ & $\sigma_2$ & $\sigma_0$\\
$\sigma_2$ &  $\sigma_2$ & $\sigma_0$ & $\sigma_1$\\
\hline 
\end{tabular} 
& 
\begin{tabular}{l|lll}
$\Sigma_7$ & $\sigma_0$ & $\sigma_1$ & $\sigma_2$ \\
\hline 
$\sigma_0$ &   $\sigma_0$ & $\sigma_1$ & $\sigma_2$\\
$\sigma_1$ &  $\sigma_1$ & $\sigma_1$ & $\sigma_1$\\
$\sigma_2$ &  $\sigma_2$ & $\sigma_1$ & $\sigma_0$\\
\hline 
\end{tabular} 
&  \\
 & & & 
\end{tabular}
\caption{Up to isomorphism, there are seven monoids with three elements~\cite{C09}. }\label{tabelamonoides}
\end{table}

\begin{rem} The eight three-cell networks with symmetry monoids with three elements have the following properties according to the number of nontrivial synchrony spaces: $A_2 \& A_1$ has no nontrivial synchrony space (from Table~\ref{tab:C3L0RI2.tex}); $E_6 \& E_4$, $C_1 \& D_1$ and $C_1 \& B_1$ have one nontrivial synchrony space (from Table~\ref{tab:C3L1.tex}); $E_6 \& F_5$, $C_1 \& E_6$ and $C_1 \& E_3$ have two nontrivial synchrony spaces (from Table~\ref{tab:C3L2.tex}); $C_1 \& C_2$ has three nontrivial synchrony spaces (from Table~\ref{tab:C2L3.tex}). 
\hfill
$\Diamond$
\end{rem}

\begin{rem}
The networks $C_1 \& E_3$ and $C_1 \& E_6$ are ODE distinct and have the same symmetry monoid. Thus they have the same fundamental network. 
Which in this case is the network with set of three cells $\Sigma = \{ \sigma_0, \sigma_1, \sigma_2\}$ and the asymmetric inputs can be read off from the multiplication table of $\Sigma_3$ in Table~\ref{tabelamonoides} (recall that $\tilde{\sigma}_j$ encodes the left-multiplicative behaviour of $\sigma_j$): 
$$
\widetilde{\sigma}_0 = \left[ 1\, 2\, 3\right],\quad \widetilde{\sigma}_1 = \left[ 2\, 2\, 2\right],\quad \widetilde{\sigma}_2 = \left[ 3\, 2\, 3\right]\, .
$$
In fact, this three-cell fundamental network with asymmetric inputs $\widetilde{\sigma}_1$ and $\widetilde{\sigma}_2$ corresponds to an isomorphic  network of  $C_1 \& E_3$. Thus $C_1 \& E_3$ is a fundamental network and $C_1 \& E_6$ is not. The other six networks are fundamental networks. See Table~\ref{tabelafundamentais} for the asymmetric inputs for each of the fundamental networks $\widetilde{\Sigma}_i$ associated with each of the monoids $\Sigma_i$ in Table~\ref{tabelamonoides}. 
\hfill
$\Diamond$
\end{rem}

\begin{table}
\begin{tabular}{c|c} 
\hline 
Fundamental               & Monoid                                                                    \\
 Network                      & symmetries                                                             \\
  \hline 
$\widetilde{\Sigma}_1$ & $\widetilde{\sigma}_0 = \left[ 1\, 2\, 3\right],\, \widetilde{\sigma}_1 = \left[ 2\, 2\, 2\right],\, \widetilde{\sigma}_2 = \left[ 3\, 2\, 2\right]$ \\
$\widetilde{\Sigma}_2$ & $\widetilde{\sigma}_0 = \left[ 1\, 2\, 3\right],\, \widetilde{\sigma}_1 = \left[ 2\, 2\, 3\right],\, \widetilde{\sigma}_2 = \left[ 3\, 3\, 2\right]$ \\
$\widetilde{\Sigma}_3$ & $\widetilde{\sigma}_0 = \left[ 1\, 2\, 3\right],\, \widetilde{\sigma}_1 = \left[ 2\, 2\, 2\right],\, \widetilde{\sigma}_2 = \left[ 3\, 2\, 3\right]$ \\
$\widetilde{\Sigma}_4$ & $\widetilde{\sigma}_0 = \left[ 1\, 2\, 3\right],\, \widetilde{\sigma}_1 = \left[ 2\, 2\, 3\right],\, \widetilde{\sigma}_2 = \left[ 3\, 2\, 3\right]$ \\
$\widetilde{\Sigma}_5$ & $\widetilde{\sigma}_0 = \left[ 1\, 2\, 3\right],\, \widetilde{\sigma}_1 = \left[ 2\, 2\, 2\right],\, \widetilde{\sigma}_2 = \left[ 3\, 3\, 3\right]$ \\
$\widetilde{\Sigma}_6$ & $\widetilde{\sigma}_0 = \left[ 1\, 2\, 3\right],\, \widetilde{\sigma}_1 = \left[ 2\, 3\, 1\right],\, \widetilde{\sigma}_2 = \left[ 3\, 1\, 2\right]$ \\
$\widetilde{\Sigma}_7$ & $\widetilde{\sigma}_0 = \left[ 1\, 2\, 3\right],\, \widetilde{\sigma}_1 = \left[ 2\, 2\, 2\right],\, \widetilde{\sigma}_2 = \left[ 3\, 2\, 1\right]$ \\
\hline 
\end{tabular}
\caption{The seven fundamental networks with three cells and two asymmetric inputs corresponding to the  symmetry monoids with three elements in Table~\ref{tabelamonoides}. Here, $\widetilde{\sigma}_0$ represents the dependence of each cell on its own state.} \label{tabelafundamentais}
\end{table}

\begin{rem}
More generally, Aguiar, Dias and Soares~\cite[Theorem 5.16]{ADS19b} present a set of necessary and sufficient conditions (on the topology of the network) for a network with asymmetric inputs to be a fundamental network. One of such properties is the backward connectivity of the graph (i.e., there exists a cell such that any other cell has a directed path ending in that cell). We remark that the network $C_1 \& E_6$ mentioned in the previous remark is not backward connected. 
\hfill
$\Diamond$
\end{rem}

\section{Why the number n(n-1) of inputs for an n-cell network with asymmetric inputs is special?}\label{sec:main}

As a first step towards obtaining a classification, in terms of ODE-classes, of the $n$-cell networks with asymmetric inputs, for a fixed $n$, we 
show next that for every ODE-class of  $n$-cell networks with asymmetric inputs, the minimal networks have at most $n(n-1)$ asymmetric inputs. 

Given a positive integer $n$, consider the  $n^2$-dimensional real  linear space of the $n \times n$ matrices $M_{n \times n} (\R)$  with the usual operations of sum of matrices and scalar product of matrices by reals. Denote by $V_{1,n}$, the subspace of $M_{n \times n} (\R)$  generated by the valency one $n \times n$ matrices  (with integer entries $0,1$).

\begin{thm}\label{thm:dimV1,n}
For $n\geq1$, the dimension of the linear subspace $V_{1,n}$ of  $M_{n \times n} (\R)$ is $n(n-1)+1$. 
\end{thm}

\begin{proof} 
Let $d_n = n(n-1)+1$. We show that $V_{1,n}$ has dimension $d_n$. Note that $M_{n\times n}(\R)$ has dimension $n^2$. We first observe that $V_{1,n}$ has dimension at most $d_n$. There are $N= n^n$ valency one square matrices of order $n$, say $B_1, \ldots, B_{N}$. Using the isomorphism between $M_{n\times n}(\R)$ and $\R^{n^2}$ mapping $A = [a_{ij}]$ to the column vector $(a_{11}, \ldots, a_{1n}, \ldots, a_{n1}, \ldots, a_{nn})^t$,  take the $n^2 \times N$ matrix $B$ whose columns correspond to those  $N$ matrices. It follows that, the sum of the $n$ first rows of $B$ is the row $(1\,  1\,  \cdots \,  1)$, and the same row sum is obtained for the following  groups  each with $n$ rows. Thus the rank of the matrix $B$ is at most $d_n$. We show now that there are indeed $d_n$ linearly independent matrices $B_i$. There is a specific choice of valency one adjacency matrices $B_i$, such that we get the $n^2 \times d_n$ submatrix $\overline{B}$ of $B$ with the following block structure: 
$$
{\tiny 
\overline{B} = 
\left[
\begin{array}{l|l|l|l|l|l}
\mbox{Id}_n & L_1 & L_1 & \cdots & L_1 & L_1\\
\hline 
L_2 & I & L_1 &  \cdots & L_1 & L_1\\
\hline 
L_2 & L_1 & I &  \cdots & L_1 & L_1\\
\hline 
\vdots & \vdots & \vdots & \ddots & \vdots & \vdots \\
\hline 
L_2 & L_1 & L_1 &  \cdots & I & L_1 \\
\hline 
L_2 & L_1 & L_1 &  \cdots & L_1 & I
\end{array}
\right]\, . }
$$
Here the blocks $I,\, L_1$ are $n \times (n-1)$  and $L_2$ is $n \times n$ having the form:
$$
I = 
\left[
\begin{array}{c}
\mbox{Id}_{n-1}\\
0_{1,n-1}
\end{array}
\right],\qquad 
L_1 = \left[
\begin{array}{c}
0_{n-1, n-1}\\
1_{1,n-1}
\end{array}
\right],\qquad 
L_2 = \left[
\begin{array}{c}
0_{n-1, n}\\
1_{1,n}
\end{array}
\right]\, .
$$
  Using the elementary operations on the columns of $\overline{B}$,  for $i=n+1, \ldots, d$, replacing the column $C_i$ by $C_i - C_n$, we obtain the matrix:
  $$
  {\tiny 
S = \left[
\begin{array}{l|l|l|l|l|l}
\mbox{Id}_n & 0_{n,n-1}  & 0_{n,n-1}  & \cdots & 0_{n,n-1}  & 0_{n,n-1}  \\
\hline 
L_2 & I^* & 0_{n,n-1}  &  \cdots & 0_{n,n-1}  & 0_{n,n-1} \\
\hline 
L_2 & 0_{n,n-1}  & I^* &  \cdots & 0_{n,n-1}  &0_{n,n-1} \\
\hline 
\vdots & \vdots & \vdots & \ddots & \vdots & \vdots \\
\hline 
L_2 & 0_{n,n-1}  & 0_{n,n-1}   &  \cdots & I^* & 0_{n,n-1}  \\
\hline 
L_2 & 0_{n,n-1}  & 0_{n,n-1} &  \cdots &0_{n,n-1}  & I^*
\end{array}
\right]}
$$
where
$$
I^*= 
\left[
\begin{array}{c}
\mbox{Id}_{n-1}\\-1_{1,n-1}
\end{array}
\right]\, .
$$
Clearly, the rank of $S$ is $n + (n-1)(n-1)$, that is, $d_n= n(n-1)+1$. 
\end{proof}

\begin{exam} \label{ex:dimVn3}
To illustrate the above result, we consider the 3-cell networks with asymmetric inputs. As we have showed, the dimension $d_3$ of the linear subspace $V_{1,3}$  of $M_{3,3}(\real)$, generated by the valency one $3\times 3$ matrices, is 7. We take the following $3 \times 3$ valency one matrices:\\
\begin{center}
{\tiny \begin{tabular}{l}
$M_1 = \left[
\begin{array}{ccc}
1 & 0 & 0  \\
0 & 0 & 1 \\
0 & 0 & 1
\end{array}
\right],\, 
M_2 =\left[
\begin{array}{ccc}
0 & 1 & 0  \\
0 & 0 & 1 \\
0 & 0 & 1
\end{array}
\right],\, 
M_3 = \left[
\begin{array}{ccc}
0 & 0 & 1  \\
0 & 0 & 1 \\
0 & 0 & 1
\end{array}
\right],\, 
M_4 = \left[
\begin{array}{ccc}
0 & 0 & 1  \\
1 & 0 & 0 \\
0 & 0 & 1
\end{array}
\right],$ \\
\ \\
$M_5 = \left[
\begin{array}{ccc}
0 & 0 & 1  \\
0 & 1 & 0 \\
0 & 0 & 1
\end{array}
\right],\, 
M_6 = \left[
\begin{array}{ccc}
0 & 0 & 1  \\
0 & 0 & 1 \\
1 & 0 & 0
\end{array}
\right],\, 
M_7 = \left[
\begin{array}{ccc}
0 & 0 & 1  \\
0 & 0 & 1 \\
0 & 1 & 0
\end{array}
\right]\, .$
\end{tabular}}
\end{center}
(i) Using the isomorphism $M_{3 \times 3}(\R) \to \R^9$ mapping $A = [a_{ij}]$ to the column vector $\left(a_{11}\, a_{12}\, \ldots\,  a_{33}\right)^t$, we can form 
the $9 \times 7$ matrix whose columns correspond to the above 7 matrices:
$$
{\tiny 
\left[
\begin{array}{ccc|cc|cc}
1 & 0 & 0  & 0 & 0 & 0 & 0  \\
0 & 1 & 0  & 0 & 0 & 0 & 0  \\
0 & 0 & 1  & 1 & 1 & 1 & 1  \\
\hline 
0 & 0 & 0  & 1 & 0 & 0 & 0  \\
0 & 0 & 0  & 0 & 1 & 0 & 0  \\
1 & 1 & 1  & 0 & 0 & 1 & 1  \\
\hline 
0 & 0 & 0  & 0 & 0 & 1 & 0  \\
0 & 0 & 0  & 0 & 0 & 0 & 1  \\
1 & 1 & 1  & 1 & 1 & 0 & 0  
\end{array}
\right]\, .}
$$
This matrix is the submatrix $\overline{B}$  in the proof of Theorem~\ref{thm:dimV1,n}, when $n=3$.
 Thus the matrices $M_1, \ldots, M_7$ form a basis of $V_{1,3}$.\\
(ii) Consider now the seven 3-cell networks with one (asymmetric) input and adjacency matrices $M_1, \ldots, M_7$, say $G_1, \ldots, G_7$, respectively. We have that $[G_1] = [G_5]$, $[G_2] = [G_4]$ and $[G_6] = [G_7]$, and that $G_1, G_2, G_3, G_6$ are minimal representatives of four distinct $ODE$-classes. \\
(iii)
We have
$$
A = \left[
\begin{array}{ccc}
0 & 0 & 1  \\
1 & 0 & 0 \\
0 & 1 & 0
\end{array}
\right] = -M_3 + M_4 + M_7 
$$
and so $\{ A, M_1, M_2, M_3, M_4, M_5, M_6\}$ is also a basis of $V_{1,3}$. Similarly, we have
$$B = \left[
\begin{array}{ccc}
1 & 0 & 0  \\
0 & 0 & 1 \\
0 & 1 & 0
\end{array}
\right]= A+ M_1 - M_4.$$
 Thus $\{ A, B, M_1, M_2, M_3, M_5, M_6\}$ is also a basis of $V_{1,3}$. Finally, we have that 
$\mbox{id}_3 = M_1 -M_3 + M_5$. We get then that $\{ \mbox{id}_3, A, B, M_1, M_2, M_3, M_6\}$ is also a basis of $V_{1,3}$. 
We saw  in Lemma~\ref{lem:um3} that $ A, B, M_1, M_2, M_3, M_6$ are adjacency matrices of representatives of the (six) distinct ODE-classes of the $3$-cell networks with one (asymmetric) input. 
\hfill $\Diamond$
\end{exam}

\begin{thm}\label{thm:dim}
If $G$ is an  $n$-cell  network with $k$ asymmetric inputs which is minimal then  $k \le n(n-1)$.
\end{thm}

\begin{proof}
By the previous theorem, $V_{1,n}$ has dimension $d_n= n^2 - (n-1)$. The result  follows trivially, as if $G$ is an $n$-cell network with $k$ asymmetric inputs given by the valency one adjacency matrices $A_1, \ldots, A_k$, by Proposition~\ref{prop:min}, $G$ is minimal if and only if the matrices $Id_n, A_1, \ldots, A_k$ are linearly independent. Thus, in particular, $A_i \ne \mbox{Id}_n$, for $i=1,\ldots, k$ and $k$ is at most $d_n-1 = n(n-1)$. 
\end{proof}

\begin{cor}\label{cor:minimal}
An $n$-cell network with asymmetric inputs is ODE-equivalent to an $n$-cell network with at most $n(n-1)$ asymmetric inputs.
\end{cor}

We have then that if $G$ is an $n$-cell minimal network  with $m$ asymmetric inputs then  $m \leq n(n-1)$. In particular, we have that for all $k > n(n-1)$, 
$$
\mbox{Min}_{k,n} = \emptyset\, .
$$

As remarked before,  if there is no restriction on the  inputs, then the number of distinct ODE-classes  of $n$-cell networks is not finite. However, if we restrict to networks with asymmetric inputs, as the number of $n$-cell networks with asymmetric inputs with at most $n(n-1)$ asymmetric inputs is finite, it also follows from 
Corollary~\ref{cor:minimal} that:

\begin{thm}
 The number of distinct ODE-classes of $n$-cell networks with asymmetric inputs is finite. 
\end{thm}

\begin{exam} 
Consider the set of 2-cell networks with asymmetric inputs. We have by Corollary~\ref{cor:minimal} that any such network is ODE-equivalent to a $2$-cell network with at most $2$ asymmetric inputs. Moreover, by Theorem~\ref{thm:dim}, the linear space  $V_{1,2}$ generated by the $2 \times 2$ valency one matrices (with integer entries $0,1$) is $3$. For example $\mbox{Id}_2$ and 
$$A_1=\left[
\begin{array}{cc}
1 & 0\\
1 & 0
\end{array}
\right],\quad A_2=\left[
\begin{array}{cc}
0 & 1\\
1 & 0
\end{array}
\right]
$$
form a basis of $V_{1,2}$. 
We can check that, up to ODE-equivalence, there are only 4 classes of 2-cell networks with asymmetric inputs, with the following representative networks: the 2-cell network with no inputs, two 1-input networks given by $A_1$ and $A_2$, and one network with two asymmetric inputs given by $A_1,A_2$. 

\hfill $\Diamond$
\end{exam}

\section{The ODE-class of the n-cell networks with n(n-1) asymmetric inputs}\label{sec:universal}

In this section, we start by observing that there is a unique ODE-class of the $n$-cell networks with $n(n-1)$ asymmetric inputs. We then address the issue of finding a minimal representative of that class.

As a direct consequence of Theorem~\ref{thm:dimV1,n}, we have that:

\begin{cor} 
If $n$ is a positive integer then all networks in $\mbox{Min}_{n(n-1),n}$ are ODE-equivalent.
\end{cor}

\begin{proof} 
Given two (minimal) networks $G_1, G_2 \in \mbox{Min}_{n(n-1),n}$, with adjacency matrices $A_i$ and $B_i$, respectively, for $i=1, \ldots, n(n-1)$, we have that $\mbox{id}_n, A_1, \ldots, A_{n(n-1)}$ are linearly independent. Similarly, $\mbox{Id}_n, B_1, \ldots, B_{n(n-1)}$ are linearly independent. Thus both sets form a basis of $V_{1,n}$, that is, $G_1$ and $G_2$ are ODE-equivalent.
\end{proof}

Given an $n$-cell network $G$ with adjacency matrix $A_G$ and given a permutation $\pi \in {\bf S}_n$ on its set of cells $\{ 1, \ldots, n\}$, we denote by $\pi G$  the network obtained from $G$ by permuting the cells according to $\pi$. Thus the adjacency matrix  of $\pi G$ is $P_{\pi}^{-1} A _G P_{\pi}$, where $P_\pi$ is the permutation matrix corresponding to $\pi$.

Note that any representative of the ODE-class  $\mbox{Min}_{n(n-1),n}$ is the union network of $n(n-1)$ networks in  $\mbox{Min}_{1,n}$.
In the next section, we show that $\mbox{Min}_{1,n}$ has at least $n(n-1)$ distinct ODE-classes. 
It might seem natural that selecting any network in each of those classes, then their union would be a minimal network in $\mbox{Min}_{n(n-1),n}$.
The following example shows that this depends on the networks in $\mbox{Min}_{1,n}$ that we select.

\begin{exam} 
Fix $n=3$ and consider the six distinct ODE-classes of the $3$-cell minimal networks with one (asymmetric) input, given by Lemma~\ref{lem:um3}. As remarked in Example~\ref{ex:dimVn3} the adjacency matrices of the representatives of the six distinct ODE-classes, the networks $A, \ldots, F$ in Table~~\ref{tab:val13cell}  are linearly independent together with the identity matrix $\mbox{id}_3$. Thus, the union of the six networks $A, \ldots, F$ is a minimal network in $\mbox{Min}_{6,3}$. 
However, if we consider instead the representatives $A$, $\pi_1 B, \pi_2 C, \pi_3 D, \pi_3 E, \pi_2 F$, where $\pi_1,\pi_2,\pi_3$, are the cell permutations given by $\pi_1 = (321)$, $\pi_2 = (213)$ and $\pi_3 = (312)$, then the subspace generated by the corresponding adjacency matrices, together with $\mbox{id}_3$, has dimension $4 \ne 7$. It follows that the union of the networks $A$, $\pi_1 B, \pi_2 C, \pi_3 D, \pi_3 E, \pi_2 F$ is not a minimal network in $\mbox{Min}_{6,3}$.
\hfill 
$\Diamond$
\end{exam}

For the case $n=4$, if we select randomly twelve ODE-distinct classes in $\mbox{Min}_{1,4}$, we have noticed that we will not obtain a representative of $\mbox{Min}_{12,4}$. 
However, the next example shows that, we may find two such representatives of $\mbox{Min}_{12,4}$ using distinct ODE-classes in $\mbox{Min}_{1,4}$. 

\begin{exam} 
We present below two $16\times 13$ matrices with rank $13$, each corresponding to a different choice of subsets  of $\mbox{Min}_{1,4}$.  The first column of each of those matrices corresponds to the matrix $\mbox{id}_4$ and the other twelve columns to the adjacency matrices of the networks in the corresponding subset: 
$$
{\tiny 
\left[
\begin{array}{c|ccccc|cccc|cc|c}
1&1&1&1&1&1&0&0&0&0&0&0&0 \\
0&0&0&0&0&0&1&1&1&1&0&0&0 \\
0&0&0&0&0&0&0&0&0&0&1&1&0 \\
0&0&0&0&0&0&0&0&0&0&0&0&1 \\
\hline
0&1&0&1&1&0&1&1&1&1&1&1&1 \\
1&0&0&0&0&0&0&0&0&0&0&0&0 \\
0&0&1&0&0&0&0&0&0&0&0&0&0 \\
0&0&0&0&0&1&0&0&0&0&0&0&0 \\
\hline
0&0&1&1&1&1&0&0&0&0&0&0&0 \\
0&0&0&0&0&0&1&0&0&1&1&1&1 \\
1&1&0&0&0&0&0&1&0&0&0&0&0 \\
0&0&0&0&0&0&0&0&1&0&0&0&0 \\
\hline
0&0&0&0&1&0&0&0&0&0&0&0&0 \\
0&0&0&0&0&0&0&0&0&0&0&1&0 \\
0&0&0&0&0&1&0&0&1&1&0&0&1 \\
1&1&1&1&0&0&1&1&0&0&1&0&0
\end{array}
\right],\ 
\left[
\begin{array}{c|ccccc|cccc|cc|c}
1&1&1&1&1&1&0&0&0&0&0&0&0 \\
0&0&0&0&0&0&1&1&1&1&0&0&0 \\
0&0&0&0&0&0&0&0&0&0&1&1&0 \\
0&0&0&0&0&0&0&0&0&0&0&0&1 \\
\hline
0&1&0&0&0&1&1&1&1&1&1&1&1 \\
1&0&0&0&1&0&0&0&0&0&0&0&0 \\
0&0&0&1&0&0&0&0&0&0&0&0&0 \\
0&0&1&0&0&0&0&0&0&0&0&0&0 \\
\hline
0&0&1&0&1&1&0&0&0&1&0&0&0 \\
0&0&0&0&0&0&1&1&0&0&1&1&1 \\
1&1&0&0&0&0&0&0&0&0&0&0&0 \\
0&0&0&1&0&0&0&0&1&0&0&0&0 \\
\hline
0&0&1&1&0&1&1&0&0&0&0&0&0 \\
0&0&0&0&0&0&0&1&0&0&0&1&0 \\
0&0&0&0&1&0&0&0&1&1&0&0&1 \\
1&1&0&0&0&0&0&0&0&0&1&0&0
\end{array}
\right]
}
$$
\hfill 
$\Diamond$
\end{exam}

On the other hand, it would be unexpected that the union of a set in $\mbox{Min}_{1,n}$ of networks where most are ODE-equivalent would be a network representative of $\mbox{Min}_{n(n-1),n}$. We  prove indeed that one such representative is  given by considering $n-1$ feed-forward networks and their orbits under the cyclic permutation group on the $n$ cells.

\begin{lemma} \label{lemma:FFN_tails}
Given $n \in \NN$, up to permutation of the cells, the number of $n$-cell feed-forward networks with one (asymmetric) input having at most one tail with length greater than one is $n-1$.
\end{lemma}

\begin{proof}
An $n$-cell feed-forward network with one (asymmetric) input having at most one tail with length greater than one satisfies one of the following: it has $n-1$ tails with length one, it has $n-3$ tails with length one and one tail with length two,..., it has one tail with length one and one tail with length  $n-2$ or has only one tail with length $n-1$. Thus, up to permutation of the cells, there are $n-1$ such networks.
\end{proof}

 Denote by $\ZZ_n$ the cyclic subgroup of ${\bf S}_n$ generated by the $n$-cycle permutation $\pi_n = (1\, 2\, \cdots \, n)$. Let 
$$
{\small 
\begin{array}{rclrcl} 
\ZZ_n G & =  & \left\{ \pi_n^j G:\, j=0, 1, \ldots, n-1\right\}, &
\ZZ_n A_G & =  & \left\{ P_{\pi_n^j}^{-1} A _G P_{\pi_n^j}:\, j=0, 1, \ldots, n-1\right\} \, .
\end{array}}
$$

\begin{thm} \label{cor:rep_min}
Given $n \in \NN$, consider the $n-1$ feed-forward networks, $\mathsf{F}_1,\,  \mathsf{F}_2, \, \ldots,$ $\mathsf{F}_{n-1}$,  with $n-1, \, n-3,\, n-4,\,  \ldots, 0$ length one tails, respectively, as in Lemma~ \ref{lemma:FFN_tails}. 
The $n$-cell network with $n(n-1)$ asymmetric inputs given by the union 
$$
\bigcup_{i=1}^{n-1} \, \ZZ_n \mathsf{F}_i 
$$
is a representative of the minimal class $\mbox{Min}_{n(n-1),n}$.
\end{thm}

\begin{proof}
Consider the $n-1$ feed-forward networks in the conditions of Lemma~ \ref{lemma:FFN_tails}, $\mathsf{F}_1, \, \mathsf{F}_2, \, \ldots, \, \mathsf{F}_{n-1}$,  with $n-1, \, n-3,\, n-4,\,   \ldots,\,  0$ length one tails. Without loss of generality, we can consider that the cells in each $\mathsf{F}_i$ are enumerated such that the cell $1$ receives a self-input, the cells $2,\dots,n-i+1$ receive an edge from cell $1$ and the cells $n-i+2,\dots, n$, when $i>1$, receive an edge, respectively, from the cells $n-i+1,\dots, n-1$.

Consider the matrix $B$ whose columns $1+(i-1)n, \ldots, n+(i-1)n$ correspond to the matrices in the sets $\ZZ_n A_{\mathsf{F}_i}$, 
for $i=1,\ldots,n-1$, by row. We have that, the rows $1+(i-1)n, \ldots, n+(i-1)n$, for $i=1,\ldots,n$, of $B$, correspond to the inputs that cell $i$ receives from cells $1,\ldots,n$, respectively, in the networks $\ZZ_n \mathsf{F}_i$ for $i=1,\ldots,n-1$.

We have the following observations: among the networks $\ZZ_n \mathsf{F}_i$, 
for $i=1,\ldots,n-1$, 
 there is only one network, $\mathsf{F}_1$, such that cell $n$ receives its input from cell $1$. Thus, there is one row of $B$ with the entry in the first column equal to 1 and all the other entries equal to $0$. Using the permutations in $\ZZ_n$, there is one row of $B$ with the entry in column $k$ equal to 1 and all the other entries equal to $0$, for $k=2,\ldots,n$, Among the networks $\ZZ_n \mathsf{F}_i$, for $i=1,\ldots,n-1$,  
there are only two networks, $\mathsf{F}_1$ and $\mathsf{F}_2$, such that cell $n-1$ receives its input from cell $1$. Thus, there is one row of $B$ with the entries in columns 1 and $n+1$  equal to 1 and all the other entries equal to $0$. Using  the permutations in $\ZZ_n$, 
for $k=2,\ldots,n$, there is one row of $B$ with the entries in columns $k$ and $(k+n)$ equal to 1 and all the other entries equal to $0$. 
This reasoning applies recursively, until cell $3$. Among the networks $\ZZ_n \mathsf{F}_i$, for $i=1,\ldots,n-1$, 
there are only $n-2$ networks, $\mathsf{F}_i$, $i=1,\ldots,n-2$, such that cell $3$ receives its input from cell $1$. Thus, there is one row of $B$ with the entries in columns $1+ (i-1)n$, for $i=1,\ldots,n-2$, equal to 1 and all the other entries equal to $0$. Using  the permutations in $\ZZ_n$, 
for each $k=2,\ldots,n$, there is one row of $B$ with the entries in the columns $k+ (i-1)n$, for $i=1,\ldots,n-2$, equal to 1 and all the other entries equal to $0$. Finally, for the cell $1$, we have that, among the networks $\ZZ_n \mathsf{F}_i$, for $i=1,\ldots,n-1$, 
there are only $n-1$ networks, $\mathsf{F}_i$, $i=1,\ldots,n-1$, such that cell $1$ has a self-loop. Thus, there is one row of $B$ with the entries in columns $1+ (i-1)n$, for $i=1,\ldots,n-1$, equal to 1 and all the other entries equal to $0$. Using  the permutations in $\ZZ_n$, 
for each $k=2,\ldots,n$, there is one row of $B$ with the entries in the columns $k+ (i-1)n$, for $i=1,\ldots,n-1$, equal to 1 and all the other entries equal to $0$.

Taking the above observations into account, we conclude that there is a permutation of the rows of matrix $B$ such that $B$ is row-equivalent to a matrix with the following lower triangular block form:
$$
{\tiny 
\left[
\begin{array}{ccccc}
\mbox{Id}_n & 0 & 0 & \cdots & 0 \\
\mbox{Id}_n & \mbox{Id}_n & 0& \cdots & 0 \\
\cdots & \cdots & \cdots & \cdots& \cdots \\
\mbox{Id}_n  & \mbox{Id}_n  & \mbox{Id}_n  & \cdots & \mbox{Id}_n  \\
\mbox{Id}_n & B_1  & B_2 & \cdots & B_{n-2} \\
\end{array}
\right]}
$$

It follows that  the matrix $B$ has rank $(n-1)n$ and, thus, that the $(n-1)n$ matrices in  $\ZZ_n A_{\mathsf{F}_i}$, for $i=1,\ldots,n-1$, 
are linearly independent. We conclude that the network given by the union of networks, $\bigcup_{i=1}^{n-1}\, \ZZ_n \mathsf{F}_i$, is a representative of the minimal class $Min_{n(n-1),n}$.
\end{proof}

\begin{figure}
\begin{center}
\begin{tabular}{ccc}
\begin{tikzpicture}
 [scale=.15,auto=left, node distance=1.5cm, every node/.style={circle,draw}]
  \node[fill=white] (n1) at (4,0) {\small{1}};
  \node[fill=white] (n2) at (24,0) {\small{2}};
 \node[fill=white] (n3) at (14,9)  {\small{3}};
 \draw[ arrows={->}, thick]  (n1) to [loop left] (n1);
\draw[ arrows={->}, thick] (n1) edge  [bend left=-5] (n2);
\draw[ arrows={->}, thick]  (n1) edge  [bend left=5]  (n3);
 \end{tikzpicture} & 
 \begin{tikzpicture}
 [scale=.15,auto=left, node distance=1.5cm, every node/.style={circle,draw}]
  \node[fill=white] (n1) at (4,0) {\small{1}};
  \node[fill=white] (n2) at (24,0) {\small{2}};
 \node[fill=white] (n3) at (14,9)  {\small{3}};
  \draw[ arrows={->>}, thick]  (n2) to [loop right] (n2);
\draw[ arrows={->>}, thick]  (n2) edge  [bend left=25] (n1);
\draw[ arrows={->>}, thick] (n2) edge  [bend left=-5] (n3);
 \end{tikzpicture} & 
 \begin{tikzpicture}
 [scale=.15,auto=left, node distance=1.5cm, every node/.style={circle,draw}]
  \node[fill=white] (n1) at (4,0) {\small{1}};
  \node[fill=white] (n2) at (24,0) {\small{2}};
 \node[fill=white] (n3) at (14,9)  {\small{3}};
 \draw[ arrows={->>>}, thick]  (n3) to [loop above] (n3);
 \draw[ arrows={->>>}, thick]   (n3) edge  [bend left=-25] (n1);
 \draw[ arrows={->>>}, thick]  (n3) edge  [bend left=25] (n2);
 \end{tikzpicture}\\
  \ \\
 \begin{tikzpicture}
 [scale=.15,auto=left, node distance=1.5cm, every node/.style={circle,draw}]
  \node[fill=white] (n1) at (4,0) {\small{1}};
  \node[fill=white] (n2) at (24,0) {\small{2}};
 \node[fill=white] (n3) at (14,9)  {\small{3}};
 \draw[ arrows={->}, dashed] (n1) to [loop below] (n1);
\draw[ arrows={->}, dashed]  (n1) edge  [bend left=-45] (n2);
\draw[ arrows={->}, dashed]  (n2) edge  [bend left=-45] (n3);
 \end{tikzpicture} & 
 \begin{tikzpicture}
 [scale=.15,auto=left, node distance=1.5cm, every node/.style={circle,draw}]
 \node[fill=white] (n1) at (4,0) {\small{1}};
  \node[fill=white] (n2) at (24,0) {\small{2}};
 \node[fill=white] (n3) at (14,9)  {\small{3}};
 \draw[ arrows={->>}, dashed]  (n2) to [loop below] (n2);
\draw[ arrows={->>}, dashed]  (n2) edge  [bend left=-65] (n3);
\draw[ arrows={->>}, dashed]  (n3) edge  [bend left=-45] (n1);
 \end{tikzpicture}  & 
 \begin{tikzpicture}
 [scale=.15,auto=left, node distance=1.5cm, every node/.style={circle,draw}]
  \node[fill=white] (n1) at (4,0) {\small{1}};
  \node[fill=white] (n2) at (24,0) {\small{2}};
 \node[fill=white] (n3) at (14,9)  {\small{3}};
\draw[ arrows={->>>}, dashed] (n3) to [loop below] (n3);
\draw[ arrows={->>>}, dashed] (n3) edge  [bend left=-65] (n1);
\draw[ arrows={->>>}, dashed] (n1) edge  [bend left=-65] (n2);
 \end{tikzpicture} 
\end{tabular}
\caption{The networks $C_1, C_2, C_3, D_1, D_3, D_5$ from Table~\ref{tab:listaval13cell} whose union  represent a minimal network in the class $\mbox{Min}_{6,3}$.}
		\label{fig:set_srep_min}
\end{center}
\end{figure}
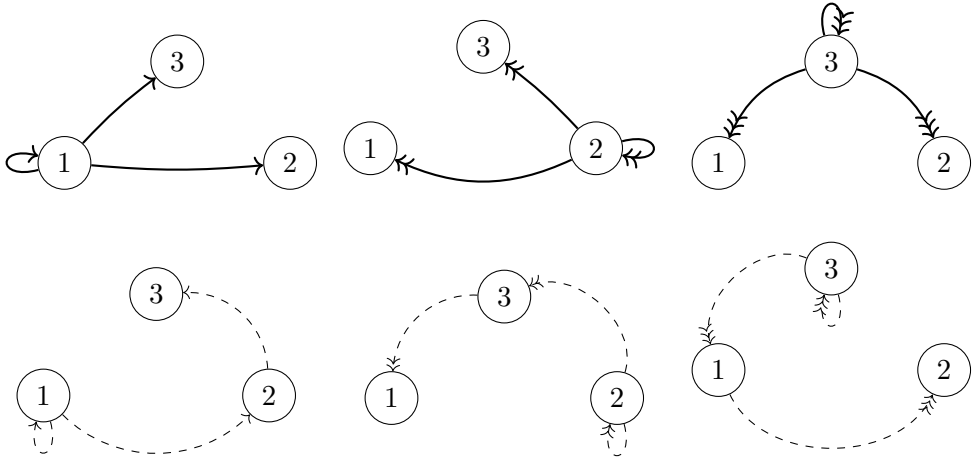

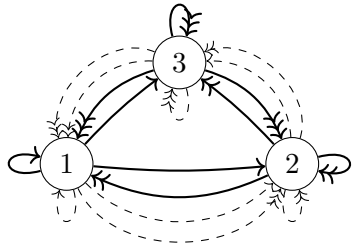
\begin{figure}[ht!]
\begin{center}
\vspace{-4mm}
\hspace{-4mm}
\begin{tikzpicture}
% [scale=.15,auto=left, node distance=1.5cm, every node/.style={circle,draw}]
[scale=.15,auto=left, node distance=1.5cm, every node/.style={circle,draw}]
 \node[fill=white] (n1) at (4,0) {\small{1}};
  \node[fill=white] (n2) at (24,0) {\small{2}};
 \node[fill=white] (n3) at (14,9)  {\small{3}};

\draw[ arrows={->}, thick]  (n1) to [loop left] (n1);
\draw[ arrows={->}, thick] (n1) edge  [bend left=-5] (n2);
\draw[ arrows={->}, thick]  (n1) edge  [bend left=5]  (n3);

 \draw[ arrows={->>}, thick]  (n2) to [loop right] (n2);
\draw[ arrows={->>}, thick]  (n2) edge  [bend left=25] (n1);
\draw[ arrows={->>}, thick] (n2) edge  [bend left=-5] (n3);

 \draw[ arrows={->>>}, thick]  (n3) to [loop above] (n3);
 \draw[ arrows={->>>}, thick]   (n3) edge  [bend left=-25] (n1);
 \draw[ arrows={->>>}, thick]  (n3) edge  [bend left=25] (n2);

 \draw[ arrows={->}, dashed] (n1) to [loop below] (n1);
\draw[ arrows={->}, dashed]  (n1) edge  [bend left=-45] (n2);
\draw[ arrows={->}, dashed]  (n2) edge  [bend left=-45] (n3);

%\draw[ arrows={->>}, dashed]  (n2) to [loop below] (n2);
\draw[ arrows={->>}, dashed]  (n2) to [in=285,out=255, loop] (n2);
\draw[ arrows={->>}, dashed]  (n2) edge  [bend left=-65] (n3);
\draw[ arrows={->>}, dashed]  (n3) edge  [bend left=-45] (n1);

\draw[ arrows={->>>}, dashed] (n3) to [loop below] (n3);
\draw[ arrows={->>>}, dashed] (n3) edge  [bend left=-65] (n1);
\draw[ arrows={->>>}, dashed] (n1) edge  [bend left=-65] (n2);
\end{tikzpicture}
\caption{A representative of the minimal class $\mbox{Min}_{6,3}$.}
		\label{fig:rep_min}
		\end{center}
\end{figure}

In the next example, we illustrate Theorem~\ref{cor:rep_min} when $n$ is equal to $3$.

\begin{exam}  
Up to permutation of the cells, there are two $3$-cell feed-forward networks with one (asymmetric) input, one having two tails of length one each  and the other having just one tail with length two. See networks $C\equiv \mathsf{F}_1$ and $D \equiv \mathsf{F}_2$, respectively, in Table~\ref{tab:val13cell}. Consider, also, the networks in the sets $\ZZ_3 \mathsf{F}_1 = \{ C_1, C_2, C_3\}$ and $\ZZ_3 \mathsf{F}_2 =\{ D_1, D_3, D_5\}$ , where the networks $C_i$ and $D_j$ appear in Table~\ref{tab:listaval13cell}. 
%We have, $C_1=C$, $C_2= \pi C$, $C_3= \pi^2 C$ and $D_1=D$, $D_3= \pi d$, $D_5= \pi^2 D$, for the permutation $\pi=(2\ 3\ 1)$.
By Theorem~\ref{cor:rep_min}, the $3$-cell network with $6$ asymmetric inputs given by the union of the networks in $\ZZ_3 \mathsf{F}_1  \cup \ZZ_3 \mathsf{F}_2$ (see Figure~\ref{fig:set_srep_min}) 
%= \{C_1,\,  C_2,\, C_3, \, D_1,\, D_3,\, D_5 \}$
is a representative of the minimal class $\mbox{Min}_{6,3}$ (see Figure~\ref{fig:rep_min}). 
\hfill 
$\Diamond$
\end{exam}

\begin{rem}\label{cor:unionmenos}
For an integer $n \ge 2$, consider the $n-1$ feed-forward networks, $\mathsf{F}_1,\,  \mathsf{F}_2, \, \ldots,$ $\mathsf{F}_{n-1}$,  with $n-1, \, n-3,\, n-4,\,  \ldots, 0$ length one tails, respectively, as in Lemma~\ref{lemma:FFN_tails}. We remark that, for each $k \leq n(n-1)$, the 
$
\left(
\begin{array}{c}
n(n-1)\\
k
\end{array}
\right)$
networks with $k$ asymmetric inputs defined by the union of the possible combinations of $k$ networks in the set $\{\ZZ_n \mathsf{F}_i:\ i=1,\ldots, n-1 \}$ are minimal networks representing ODE-classes.
However, they can represent the same ODE-classes.
For example, the minimal networks $\mathsf{F}_1$ and $\pi_n\mathsf{F}_1$ represent the same ODE-class.
Therefore the number of distinct ODE-classes in $\mbox{Min}_{1,n}$ given by those feed-forward networks is $n-1$. 
\hfill 
$\Diamond$
\end{rem}

Nevertheless, we believe that the number of distinct classes in  $\mbox{Min}_{k,n}$ can be lower bounded by the number of $k$-combinations from $n(n-1)$ elements.
In the next section, we prove that this lower bound is valid when $k=1$.

\section{More on n-cell networks with one asymmetric input}\label{sec:lower_min_1_input}

Observe that for $n=3$ and $n=2$,  the number of distinct ODE-classes of the network set $\mbox{Min}_{1,n}$ is equal to $n(n-1)$. 
However, this is not true for $n\geq4$. For example, from the results obtained by \cite{AS05}, we have that, up to permutation of cells, the set $\mbox{Min}_{1,4}$ contains $18$ networks and $n(n-1)=12$ when $n=4$. More generally, from the results of \cite{AS05}, we have that the number of distinct ODE-classes in $\mbox{Min}_{1,n}$ increases quite fast with $n$ and it is bigger than $n(n-1)$ for $n \geq 4$. We present below an algorithm that provides $n(n-1)$ networks belonging to $n(n-1)$ distinct ODE-classes in $\mbox{Min}_{1,n}$ constructed from networks in distinct ODE-classes in $\mbox{Min}_{1,l}$ for $l <n$.

\begin{table}
 \begin{center}
 {\tiny 
 \begin{tabular}{|c|c||c|c|}
\hline 
Three-cell Network &  Four-cell Network & Three-cell Network &  Four-cell Network \\
\hline 
\begin{tikzpicture}
 [scale=.15,auto=left, node distance=1.5cm, every node/.style={circle,draw}]
 \node[fill=white] (n1) at (4,0) {\small{1}};
  \node[fill=white] (n2) at (14,0) {\small{2}};
 \node[fill=white] (n3) at (14,10)  {\small{3}};
 \draw[->, thick] (n1) edge  [bend left=-10] (n2); 
 \draw[->, thick] (n2) edge  [bend left=-10] (n3); 
\draw[->, thick] (n3) edge [bend right=10] (n1); 
\end{tikzpicture} & 
\begin{tikzpicture}
 [scale=.15,auto=left, node distance=1.5cm, every node/.style={circle,draw}]
 \node[fill=white] (n1) at (4,0) {\small{1}};
  \node[fill=white] (n2) at (14,0) {\small{2}};
 \node[fill=white] (n3) at (14,10)  {\small{3}};
 \node[fill=white] (n4) at (4,10)  {\small{4}};
 \draw[->, thick] (n1) edge  [bend left=-10] (n2); 
 \draw[->, thick] (n2) edge  [bend left=-10] (n3); 
\draw[->, thick] (n3) edge [bend right=10] (n4); 
\draw[->, thick] (n4) edge [bend right=10] (n1); 
\end{tikzpicture} & 
\begin{tikzpicture}
 [scale=.15,auto=left, node distance=1.5cm, every node/.style={circle,draw}]
 \node[fill=white] (n1) at (4,0) {\small{1}};
  \node[fill=white] (n2) at (14,0) {\small{2}};
 \node[fill=white] (n3) at (14,10)  {\small{3}};
 
 \draw[->, thick] (n1) to [loop above] (n1); 
 \draw[->, thick] (n2) edge  [bend left=-10] (n3); 
\draw[->, thick] (n3) edge [bend right=10] (n2); 
\end{tikzpicture} & 
\begin{tikzpicture}
 [scale=.15,auto=left, node distance=1.5cm, every node/.style={circle,draw}]
 \node[fill=white] (n1) at (4,0) {\small{1}};
  \node[fill=white] (n2) at (14,0) {\small{2}};
 \node[fill=white] (n3) at (14,10)  {\small{3}};
 \node[fill=white] (n4) at (4,10)  {\small{4}};
 \draw[->, thick] (n1) to [loop above] (n1); 
 \draw[->, thick] (n2) edge  [bend left=-10] (n3); 
\draw[->, thick] (n3) edge [bend right=10] (n4); 
\draw[->, thick] (n4) edge [bend right=10] (n2); 
\end{tikzpicture} 
\\
\ &\ & \ &  \\
\hline 
\ & \ & \ & \\
\begin{tikzpicture}
 [scale=.15,auto=left, node distance=1.5cm, every node/.style={circle,draw}]
 \node[fill=white] (n1) at (4,0) {\small{1}};
  \node[fill=white] (n2) at (14,0) {\small{2}};
 \node[fill=white] (n3) at (14,10)  {\small{3}};
 \draw[->, thick] (n1) to [loop above] (n1); 
 \draw[->, thick] (n1) edge  [bend left=-10] (n2); 
 \draw[->, thick] (n1) edge  [bend left=-10] (n3); 
 \end{tikzpicture} & 
 \begin{tikzpicture}
 [scale=.15,auto=left, node distance=1.5cm, every node/.style={circle,draw}]
 \node[fill=white] (n1) at (4,0) {\small{1}};
  \node[fill=white] (n2) at (14,0) {\small{2}};
 \node[fill=white] (n3) at (14,10)  {\small{3}};
 \node[fill=white] (n4) at (4,10) {\small{4}};
 \draw[->, thick] (n1) to [bend left=-10] (n4); 
 \draw[->, thick] (n4) to [bend right=10] (n1); 
 \draw[->, thick] (n1) edge  [bend left=-10] (n2); 
 \draw[->, thick] (n1) edge  [bend left=-10] (n3); 
 \end{tikzpicture}  &
\begin{tikzpicture}
 [scale=.15,auto=left, node distance=1.5cm, every node/.style={circle,draw}]
 \node[fill=white] (n1) at (4,0) {\small{1}};
  \node[fill=white] (n2) at (14,0) {\small{2}};
 \node[fill=white] (n3) at (14,10)  {\small{3}};
 \draw[->, thick] (n1) to [loop above] (n1); 
 \draw[->, thick] (n1) edge  [bend left=-10] (n2); 
 \draw[->, thick] (n2) edge  [bend left=-10] (n3); 
 \end{tikzpicture} &
 \begin{tikzpicture}
 [scale=.15,auto=left, node distance=1.5cm, every node/.style={circle,draw}]
 \node[fill=white] (n1) at (4,0) {\small{1}};
  \node[fill=white] (n2) at (14,0) {\small{2}};
 \node[fill=white] (n3) at (14,10)  {\small{3}};
 \node[fill=white] (n4) at (4,10) {\small{4}};
 \draw[->, thick] (n1) to [bend left=-10] (n4); 
 \draw[->, thick] (n4) to [bend right=10] (n1); 
 \draw[->, thick] (n1) edge  [bend left=-10] (n2); 
 \draw[->, thick] (n2) edge  [bend left=-10] (n3); 
 \end{tikzpicture} \\
 \ & & & \\
\hline 
\ & & & \\
\begin{tikzpicture}
 [scale=.15,auto=left, node distance=1.5cm, every node/.style={circle,draw}]
 \node[fill=white] (n1) at (4,0) {\small{1}};
  \node[fill=white] (n2) at (14,0) {\small{2}};
 \node[fill=white] (n3) at (14,10)  {\small{3}};
 \draw[->, thick] (n1) to [loop above] (n1); 
 \draw[->, thick] (n2) to [loop above] (n2); 
\draw[->, thick] (n1) edge [bend right=10] (n3); 
\end{tikzpicture} &
\begin{tikzpicture}
 [scale=.15,auto=left, node distance=1.5cm, every node/.style={circle,draw}]
 \node[fill=white] (n1) at (4,0) {\small{1}};
  \node[fill=white] (n2) at (14,0) {\small{2}};
 \node[fill=white] (n3) at (14,10)  {\small{3}};
 \node[fill=white] (n4) at (4,10) {\small{4}};
 \draw[->, thick] (n1) to [bend left=-10] (n4); 
 \draw[->, thick] (n4) to [bend right=10] (n1); 
 \draw[->, thick] (n2) to [loop above] (n2); 
\draw[->, thick] (n1) edge [bend right=10] (n3); 
\end{tikzpicture} &
\begin{tikzpicture}
 [scale=.15,auto=left, node distance=1.5cm, every node/.style={circle,draw}]
 \node[fill=white] (n1) at (4,0) {\small{1}};
  \node[fill=white] (n2) at (14,0) {\small{2}};
 \node[fill=white] (n3) at (14,10)  {\small{3}};
 \draw[->, thick] (n1) edge  [bend left=-10] (n2); 
 \draw[->, thick] (n2) edge  [bend left=-10] (n1); 
\draw[->, thick] (n2) edge [bend right=10] (n3); 
\end{tikzpicture} &
\begin{tikzpicture}
 [scale=.15,auto=left, node distance=1.5cm, every node/.style={circle,draw}]
 \node[fill=white] (n1) at (4,0) {\small{1}};
  \node[fill=white] (n2) at (14,0) {\small{2}};
 \node[fill=white] (n3) at (14,10)  {\small{3}};
  \node[fill=white] (n4) at (4,10) {\small{4}};
 \draw[->, thick] (n1) edge  [bend left=-10] (n2); 
 \draw[->, thick] (n2) edge  [bend right=-10] (n4); 
 \draw[->, thick] (n4) edge  [bend right=-10] (n1); 
\draw[->, thick] (n2) edge [bend right=10] (n3); 
\end{tikzpicture} \\
\ & & &  \\
\hline 
\end{tabular}}
\end{center}
\caption{Six ODE distinct four-cell networks with one (asymmetric)  
input which are minimal build from the six ODE-disctint three-cell minimal networks with one (asymmetric) input.} \label{tab:prova1}
\end{table}

\begin{table}
 \begin{center}
 {\tiny 
 \begin{tabular}{|cc|}
\hline 
\begin{tikzpicture}
 [scale=.15,auto=left, node distance=1.5cm, every node/.style={circle,draw}]
 \node[fill=white] (n1) at (4,0) {\small{1}};
  \node[fill=white] (n2) at (14,0) {\small{2}};
 \node[fill=white] (n3) at (24,0)  {\small{3}};
 \draw[->, thick] (n1) to [loop above] (n1); 
 \draw[->, thick] (n1) edge  [] (n2); 
 \draw[->, thick] (n2) edge  [] (n3); 
 \end{tikzpicture} & 3-cell feed-forward network  \\
 \ &  \\
\hline 
\ &  \\
\begin{tikzpicture}
 [scale=.15,auto=left, node distance=1.5cm, every node/.style={circle,draw}]
 \node[fill=white] (n1) at (4,0) {\small{1}};
  \node[fill=white] (n2) at (14,0) {\small{2}};
 \node[fill=white] (n3) at (14,-10)  {\small{3}};
 \node[fill=white] (n4) at (4,-10)  {\small{4}};
 \draw[->, thick] (n1) to [loop above] (n1); 
 \draw[->, thick] (n1) edge  [] (n2); 
 \draw[->, thick] (n2) edge  [] (n3); 
 \draw[->, thick] (n3) edge  [] (n4); 
 \end{tikzpicture} & 
\begin{tikzpicture}
 [scale=.15,auto=left, node distance=1.5cm, every node/.style={circle,draw}]
 \node[fill=white] (n1) at (4,0) {\small{1}};
  \node[fill=white] (n2) at (14,0) {\small{2}};
 \node[fill=white] (n3) at (14,-10)  {\small{3}};
 \node[fill=white] (n4) at (4,-10)  {\small{4}};
 \draw[->, thick] (n1) to [loop above] (n1); 
 \draw[->, thick] (n1) edge  [] (n2); 
 \draw[->, thick] (n2) edge  [] (n3); 
 \draw[->, thick] (n2) edge  [] (n4); 
\end{tikzpicture} \\
\ &  \\
\begin{tikzpicture}
 [scale=.15,auto=left, node distance=1.5cm, every node/.style={circle,draw}]
 \node[fill=white] (n1) at (4,0) {\small{1}};
  \node[fill=white] (n2) at (14,0) {\small{2}};
 \node[fill=white] (n3) at (14,-10)  {\small{3}};
 \node[fill=white] (n4) at (4,-10)  {\small{4}};
 \draw[->, thick] (n1) to [loop above] (n1); 
 \draw[->, thick] (n1) edge  [] (n2); 
 \draw[->, thick] (n2) edge  [] (n3); 
 \draw[->, thick] (n1) edge  [] (n4); 
\end{tikzpicture} &
\begin{tikzpicture}
 [scale=.15,auto=left, node distance=1.5cm, every node/.style={circle,draw}]
\node[fill=white] (n1) at (4,0) {\small{1}};
  \node[fill=white] (n2) at (14,0) {\small{2}};
 \node[fill=white] (n3) at (14,-10)  {\small{3}};
 \node[fill=white] (n4) at (4,-10)  {\small{4}};
 \draw[->, thick] (n1) to [loop above] (n1); 
  \draw[->, thick] (n4) to [loop above] (n4); 
 \draw[->, thick] (n1) edge  [] (n2); 
 \draw[->, thick] (n2) edge  [] (n3); 
\end{tikzpicture}\\
\ &   \\
\hline 
\hline
 &   \\
\begin{tikzpicture}
 [scale=.15,auto=left, node distance=1.5cm, every node/.style={circle,draw}]
 \node[fill=white] (n1) at (4,0) {\small{1}};
  \node[fill=white] (n2) at (14,0) {\small{2}};
 \draw[->, thick] (n1) to [loop above] (n1); 
 \draw[->, thick] (n1) edge  [] (n2); 
 \end{tikzpicture} & 2-cell feed-forward network \\
 \ &  \\
\hline 
\ &  \\
\begin{tikzpicture}
 [scale=.15,auto=left, node distance=1.5cm, every node/.style={circle,draw}]
 \node[fill=white] (n1) at (4,0) {\small{1}};
  \node[fill=white] (n2) at (14,0) {\small{2}};
   \node[fill=white] (n3) at (4,-10) {\small{3}};
    \node[fill=white] (n4) at (14,-10) {\small{4}};
 \draw[->, thick] (n1) to [loop above] (n1); 
 \draw[->, thick] (n1) edge  [] (n2); 
 \draw[->, thick] (n1) edge  [] (n3); 
 \draw[->, thick] (n1) edge  [] (n4); 
 \end{tikzpicture} & 
 \begin{tikzpicture}
 [scale=.15,auto=left, node distance=1.5cm, every node/.style={circle,draw}]
 \node[fill=white] (n1) at (4,0) {\small{1}};
  \node[fill=white] (n2) at (14,0) {\small{2}};
   \node[fill=white] (n3) at (4,-10) {\small{3}};
    \node[fill=white] (n4) at (14,-10) {\small{4}};
 \draw[->, thick] (n1) to [loop above] (n1); 
 \draw[->, thick] (n1) edge  [] (n2); 
  \draw[->, thick] (n3) to [loop above] (n3); 
   \draw[->, thick] (n4) to [loop above] (n4); 
 \end{tikzpicture} \\
\hline 
\end{tabular}}
\end{center}
\caption{Six ODE distinct four-cell networks with one (asymmetric)  
input which are minimal build from a three-cell and two-cell minimal feed-forward networks with one (asymmetric) input.} \label{tab:prova2}
\end{table}

By explicit computation, we can see that $\mbox{Min}_{1,1}$ has no networks and $\mbox{Min}_{1,2}$ has two ODE distinct networks. Also, from Table~\ref{tab:val13cell}, we know that $\mbox{Min}_{1,3}$ has six distinct ODE-classes of networks. We describe  now explicitly some of the distinct ODE-classes of $\mbox{Min}_{1,n}$, for $n >3$. 

\begin{alg} \label{alg1} \  \\
\normalfont 
Input: A representative network $G$ of an  ODE-class in  $\mbox{Min}_{1,n-1}$ with adjacency matrix $A_G$, where $n > 3$ is an integer. 
Let $k$ be the number of cells of the largest cycle of the network.\\
Output: a representative network $\widetilde{G}$ with adjacency matrix $\tilde{A}_{\tilde{G}}$ of an  ODE-class in  $\mbox{Min}_{1,n}$ where $k+1$ is  the number of cells of the largest cycle of the network. \\
\begin{itemize}
\item[(i)] Choose a representative network $G \in \mbox{Min}_{1,n-1}$ and consider its adjacency matrix $A_G$ . Let $k$ be the number of cells of the largest cycle of the network. 
\item[(ii)] Re-enumerate the cells if necessary so that the matrix $A_G$  has the form:
$$A_G=\left[
\begin{matrix}
C_k & 0\\ 
B & D
\end{matrix} 
\right],$$
where $C_k$ is the adjacency matrix of the cycle $1 \rightarrow 2 \rightarrow\dots\rightarrow k\rightarrow 1$, $B$ is a $(n-1-k)\times k$ matrix and $D$ is a $(n-1-k)\times(n-1-k)$ matrix. 
\item[(iii)]
Take the network with $n$ cells by the following adjacency matrix
$$\tilde{A}_{\tilde{G}}=\left[
\begin{matrix}
C_{k+1} & 0\\
 \begin{matrix}
  B& 0
\end{matrix}  & D
\end{matrix} 
\right],$$
where $0$ is a column of zeros.
\item[(iv)] Output the network with adjacency matrix $\tilde{A}_{\tilde{G}}$. 
\end{itemize}
\hfill 
$\Diamond$
\end{alg}

\begin{prop}\label{prop_alg1} 
Algorithm~\ref{alg1} applied to a set of representatives of the distinct ODE-classes in $\mbox{Min}_{1,n-1}$, where $n>3$ is an integer,  provides a set of ODE-distinct networks in  $\mbox{Min}_{1,n}$.  
\end{prop}

\begin{proof}
We follow the notation of Algorithm~\ref{alg1}. Take two graphs $G_1$ and $G_2$ in $\mbox{Min}_{1,n-1}$ with adjacency matrices $A_1$ and $A_2$ and 
consider the two networks  in $\mbox{Min}_{1,n}$  obtained as output in Algorithm~\ref{alg1} with adjacency matrices $\tilde{A}_1$ and $\tilde{A}_2$. 
We need to check that if $\tilde{A}_1$ and $\tilde{A}_2$ define ODE-equivalent networks then $A_1$ and $A_2$ define ODE-equivalent networks.
Suppose that $\tilde{A}_1$ and $\tilde{A}_2$ define ODE-equivalent networks, i.e., there exists a permutation network $P$ such that $\tilde{A}_1 P = P \tilde{A}_2$.
Note that the largest cycle of both networks is unique and it must have the same dimension, say $k+1$. Then the permutation $P$ must permute cells in the largest cycle with cells in the largest cycle and has the following form:
$$P=\left[
\begin{matrix}
P_1 & 0\\ 
0& P_2
\end{matrix} 
\right],$$
where $P_1$ is a $(k+1)\times(k+1)$ matrix and $P_2$ is a $(n-1-k)\times(n-1-k)$ matrix. 
If $\tilde{A}_1 P = P \tilde{A}_2$, then $C_{k+1}P_1=P_1C_{k+1}$. 
By Theorem 3.1.1 of \cite{Davis}, we know that $P_1=C_{k+1}^l$ for some integer $0 \leq l\leq k$.  

Let $\hat{P}$ be the permutation matrix given by
$$\hat{P}=\left[
\begin{matrix}
C_k^l & 0\\ 
0& P_2
\end{matrix} 
\right].$$

Next, we check that $\tilde{A}_1 P = P \tilde{A}_2$ implies that $A_1 \hat{P} = \hat{P} A_2$.
It is clear that $C_k C_k^l= C_k^l C_k$ and $D_1 P_2=P_2 D_2$. We need to see that $B_1 C_k^l= P_2 B_2$.
It follows from $\tilde{A}_1 P = P \tilde{A}_2$ that 
$$([B_1|0])_{i (j-l)_{k+1}}= \sum_{a=1}^{n-k} (P_2)_{i a} (B_2)_{a j}= (P_2 B_2)_{ij}, \quad\quad j<k+1$$
$$([B_1|0])_{i (k+1-l)}=0,$$
where $0$ is a column of zeros, $(j-l)_{k+1}$ is $j-l$ module $k+1$  and $1\leq i\leq n-k$.
Then $[B_1|0]=[X|0|Y|0]$ and $[P_2B_2|0]=P_2[B_2|0]=[Y|0|X|0] $, where $X$ is a $(n-k)\times (k-l)$ matrix and $Y$ is a $(n-k)\times (l-1)$ matrix.
Thus $P_2 B_2=[Y|0|X]=B_1 C_k^l$.
\end{proof}

\begin{exam}
Table~\ref{tab:prova1} illustrates the application of Algorithm~\ref{alg1} to a set of ODE-distinct networks in $\mbox{Min}_{1,3}$  (taken from Table~\ref{tab:val13cell})  by increasing for each network the largest cycle by one cell, obtaining six ODE-distinct networks in $\mbox{Min}_{1,4}$. 
\hfill 
$\Diamond$
\end{exam}

\begin{alg} \label{alg2}\normalfont \  \\
Input: The feed-forward networks $\mathsf{F}_1$,  with $n-1$ cells (and $n-1$ layers), and $\mathsf{F}_2$, with $n-2$ cells (and $n-2$ layers), respectively, where $n >3$ is a positive integer\\
Output: $2(n-1)$ ODE-distinct feed-forward networks in $\mbox{Min}_{1,n}$.
\begin{itemize}
\item[(i)] Taking the feed-forward network $\mathsf{F}_1$ with $n-1$ cells, we can join one cell in $n-1$ different ways or we can leave it separated from $\mathsf{F}_1$, obtaining $n$  feed-forward networks with $n$ cells.
 \item[(ii)]Taking the feed-forward network $\mathsf{F}_2$ with $n-2$ cells, we can join the last two cells, since $n>3$, to the same cell (except to the last one) 
in $\mathsf{F}_2$, or we can leave them  separated from $\mathsf{F}_2$ in $n-2$ different ways.
\end{itemize}
\hfill 
$\Diamond$ 
\end{alg}

In Algorithm~\ref{alg2},  we provide $2(n-1)$ feed-forward networks in $\mbox{Min}_{1,n}$, where $n>3$ is a positive integer. Trivially, we have that: 

\begin{prop} \label{prop_alg2}
The $2(n-1)$ feed-forward networks outputted from Algorithm~\ref{alg2}, where $n>3$ is a positive integer, are  ODE-distinct. 
\end{prop}

\begin{exam}
Table~\ref{tab:prova2} illustrates the construction given in Algorithm~\ref{alg2} for the case $n=4$ providing  six ODE distinct $4$-cell networks in 
$\mbox{Min}_{1,4}$ which are feed-forward. 
\hfill
$\Diamond$
\end{exam}

\begin{thm}\label{thm:min_ncell_1input} 
Let $n$ be a positive integer. The difference between the number of distinct ODE-classes in $\mbox{Min}_{1,n}$ and in $\mbox{Min}_{1,n-1}$ is at least $2(n-1)$. 
Furthermore, the number of distinct ODE-classes in $\mbox{Min}_{1,n}$ is at least $n(n-1)$.
\end{thm}

\begin{proof}
Recall that $\mbox{Min}_{1,1}$ has no networks, $\mbox{Min}_{1,2}$ has two ODE distinct networks, and that from Table~\ref{tab:val13cell}, we see that $\mbox{Min}_{1,3}$ has six distinct ODE-classes of networks. So both assertions are true for $n\leq 3$. 
Assume now that $n >3$. From Algorithm~\ref{alg1} and Proposition~\ref{prop_alg1}, we obtain ODE-distinct networks in $\mbox{Min}_{1,n}$ from ODE-distinct networks in $\mbox{Min}_{1,n-1}$.  From Algorithm~\ref{alg2} and Proposition~\ref{prop_alg2}, we obtain $2(n-1)$ ODE-distinct feed-forward networks in $\mbox{Min}_{1,n}$.  As the feed-forward networks are not ODE-equivalent to those networks obtained from extending the largest cycle because the largest cycle of the networks has different dimension, we have proved that the difference of the number of distinct ODE-classes between $\mbox{Min}_{1,n}$ and $\mbox{Min}_{1,n-1}$ is at least $2(n-1)$ for $n>1$.\\

By induction, we assume that the number of distinct ODE-classes in $\mbox{Min}_{1,n-1}$ is greater than $(n-1)(n-2)$ and using the previous claim we see that the number of distinct ODE-classes in $\mbox{Min}_{1,n}$ is greater than $n(n-1)$. Thus, the second part of the theorem follows.
\end{proof}

\section{Final conclusions}\label{sec:concl}

In this work, we have proved that the set $\mbox{Min}_{k,n}$ is empty for $k>n(n-1)$ and that there is a unique ODE-class in $\mbox{Min}_{n(n-1),n}$.
Note that the minimal representative of the unique ODE-class in $\mbox{Min}_{n(n-1),n}$ obtained in Theorem~\ref{cor:rep_min} is given by the union of $n(n-1)$ networks in $\mbox{Min}_{1,n}$ from solely $(n-1)$ distinct ODE-classes.
Nevertheless, as we have illustrated for the case of networks with $3$ and $4$ cells, it is natural to expect that there exists a minimal representative of $\mbox{Min}_{n(n-1),n}$ such that each asymmetric input corresponds to a different ODE-class in $\mbox{Min}_{1,n}$.

Moreover, we conjecture that the union of every subset of $k$ such networks, with $k<n(n-1)$, will correspond to a minimal representative of a distinct ODE-class for the networks with $k$ asymmetric inputs. Therefore, we conjecture that the binomial coefficient, below, is a lower bound for the number of distinct ODE-classes in $\mbox{Min}_{k,n}$
$$
\left(
\begin{array}{c}
n(n-1)\\
k
\end{array}
\right) = \frac{n(n-1) !}{k! \left(n(n-1)-k\right)!},\quad\quad k<n(n-1).
$$

Nevertheless, we describe two algorithms to construct at least $n(n-1)$ distinct ODE-classes with one asymmetric input and $n$ cells.
Therefore, the conjecture above holds for $k=1$. It also holds trivially for $k=0$ and $k=n(n-1)$.
We believe that the algorithms presented here can be generalized for bigger numbers of asymmetric inputs.
Since these algorithms use the minimal representative networks with less cells, we hope that those generalized algorithms lead to a proof of the conjecture for $k\leq n(n-1)/2$. 
For the values of $k$ on the second half, we conjecture that the number of ODE-classes is symmetric.
Specifically, we conjecture that the number of ODE-classes in $\mbox{Min}_{k,n}$ is equal to the number of ODE-classes in $\mbox{Min}_{n(n-1)-k,n}$.
The orthogonality of subspaces in the space generated by all adjacency matrix can lead to a proof of this conjecture. 
Note that the binomial coefficient is strictly lower than the number of ODE-classes in the cases $(n,k)=(3,2)$ and $(n,k)=(4,1)$.

\section{Acknowledgments}
 The authors were partially supported by CMUP (UID/MAT/00144/2019), which is funded by FCT with national (MCTES) and European structural funds through the programs FEDER, under the partnership agreement PT2020. The third author, PS, was supported by Grant BEETHOVEN2 of the National Science Centre, Poland, no. 2016/23/G/ST1/04081.

\end{document}